\documentclass[11pt]{article}

\usepackage{amsmath, amssymb, amsthm, enumerate}
\usepackage{graphicx, caption, subcaption}
\usepackage{xcolor}
\usepackage{cancel} %
\usepackage[pagebackref]{hyperref}

\usepackage{cleveref}
\usepackage{xspace}

\usepackage[margin=1.75in, marginparwidth=1.3in]{geometry}

\usepackage{titlesec}
\theoremstyle{plain}
\newtheorem{thm}{Theorem}[section]
\newtheorem{prop}[thm]{Proposition}
\newtheorem{lemma}[thm]{Lemma}

\newtheorem{claim}[thm]{Claim}

\theoremstyle{definition}
\newtheorem{defn}[thm]{Definition}

\newtheorem{rmkx}[thm]{Remark}
\newenvironment{rmk}
{\pushQED{\qed}\rmkx}
{\popQED\endrmkx}

\newtheorem{examplex}[thm]{Example}
\newenvironment{example}
{\pushQED{\qed}\examplex}
{\popQED\endexamplex}

\newenvironment{proofclaim}[1][Proof]{\begin{proof}[#1]}{\end{proof}}

\newcommand{\barMgn}{\overline{\mathcal{M}}_{g,n}}

\newcommand{\C}{\mathbb{C}}

\newcommand{\R}{\mathbb{R}}

\newcommand{\M}{\mathcal{M}}
\newcommand{\calM}{\mathcal{M}}

\newcommand{\T}{\mathcal{T}}
\newcommand{\calT}{\mathcal{T}}
\newcommand{\calQ}{\mathcal{Q}}

\renewcommand{\P}{\mathbb{P}}

\newcommand\RR{\mathbb{R}}

\newcommand\CC{\mathbb{C}}

\newcommand{\Mgn}{\overline{\mathcal{M}}_{g,n}}

\newcommand{\GL}{\operatorname{GL}}

\newcommand{\TM}{Teichm\"uller\xspace}
\newcounter{bencomments}

\title{The boundary of a totally geodesic subvariety of moduli space}  
\author{Frederik Benirschke\thanks{University of Chicago, \href{mailto:benirschke.math@gmail.com}{\nolinkurl{benirschke.math@gmail.com}}.}, Benjamin Dozier\thanks{Department of Mathematics, Cornell University, \href{mailto:benjamin.dozier@cornell.edu}{\nolinkurl{benjamin.dozier@cornell.edu}}. Research supported in part by NSF Grant DMS-2247244 and the Simons Foundation.
}, John Rached\thanks{Department of Mathematics, Binghamton University, \href{mailto:jabourached@binghamton.edu}{\nolinkurl{jabourached@binghamton.edu}}.} }

\AtBeginDocument{%
   \def\MR#1{}
}

\begin{document}
\maketitle

\begin{abstract}
  We consider subvarieties $N$ of $\mathcal{M}_{g,n}$, the moduli space of genus $g$ Riemann surfaces with $n$ marked points, that are totally geodesic with respect to the Teichm\"uller metric.  The Deligne-Mumford boundary of $\mathcal{M}_{g,n}$ decomposes into strata, each of which is essentially a product of lower complexity moduli spaces -- in such spaces there is a natural notion of totally geodesic.  We show that the boundary locus of $N$ in any such stratum is itself totally geodesic.  Furthermore, we prove that each such boundary locus decomposes into prime pieces, and for each such piece the projection to each factor is locally isometric in an appropriate sense.  
\end{abstract}

\section{Introduction}
\label{sec:intro}

In this paper we study the boundary of a totally geodesic subvariety of moduli space in the Deligne-Mumford compactification.  We prove that the boundary is itself totally geodesic, in a strong sense.  These results may be useful in classifying totally geodesic subvarieties.  %
And they may aid in the study of geometric and dynamical properties of these subvarieties (in particular, counting closed geodesics on them, our original motivation).  

In our proofs, we use perspectives related to classical Teichm\"uller theory (e.g. Beltrami differentials and plumbing coordinates).  There is a $GL^+(2,\R)$-invariant subvariety in the space of quadratic differentials associated to the totally geodesic subvariety.  Boundaries of spaces of quadratic differentials play an important role, and for these we use results of Chen-Wright and Benirschke.

\subsection{Main results} 
Let $\T_{g,n}$ be the Teichm\"uller space of closed Riemann surfaces of genus $g$ with $n$ (labeled) marked points, and let $\M_{g,n}$ be the corresponding moduli space (i.e. $\T_{g,n}$ is the orbifold universal cover of $\M_{g,n}$).  Let $\T:=\T_{g_1,n_1} \times \cdots \times \T_{g_k,n_k} $ be a product of Teichm\"uller spaces.  

We will work with the Deligne-Mumford compactification $\overline\M_{g,n} = \M_{g,n} \sqcup \partial \M_{g,n}$. 
 The boundary $\partial \M_{g,n}$ admits a stratification 
$ \partial \M_{g,n}=\bigsqcup_\Gamma \Delta_{\Gamma}, $
 where $\Delta_{\Gamma}$, a \textit{boundary stratum}, parametrizes stable curves with dual graph $\Gamma$.  Each $\Delta_\Gamma$ is product of lower complexity moduli spaces $\M_{g_1,n_1} \times \cdots \times \M_{g_k,n_k}$, quotiented by a finite group \footnote{The group is the automorphism group of the graph $\Gamma$, decorated with marked points and genera of components.  This group acts on the product by a composition of permutations of factors and of labelings of marked points. \label{foot:quotient}}. 
 We will refer to any such quotient of such a product as a \emph{multi-component moduli space}, and denote it $\M$.  The corresponding product space $\M_{g_1,n_1} \times \cdots \times \M_{g_k,n_k}$ will be referred to as the \emph{product cover} of $\M$.

\begin{defn}
    \label{defn:geod}
    We say that $\gamma: \R \to \T$ is a \emph{(Teichm\"uller) geodesic} if $\gamma=(\gamma_1,\ldots,\gamma_k)$, where each $\gamma_i: \R \to \T_{g_i,n_i}$ is an isometric embedding, with respect to the standard metric on $\R$, and a rescaling of the Teichm\"uller metric on $\T_{g_i,n_i}$ by some factor $c_i>0$.  We define a geodesic in $\M$ to be the projection to $\M$ of a geodesic in $\T$.  
\end{defn}

Because of the potentially different scaling factors $c_i$, these geodesics are ``multi-speed": they can move at different speeds in the different factors of $\T$.  
\begin{rmk}
    Given $X,Y\in \T$, there is a unique geodesic in $\T$ (up to scaling and isometry reparametrization of $\R$) with $X,Y$ in the image.  This follows immediately from Teichm\"uller's existence and uniqueness theorems applied to each factor $\T_{g_i,n_i}$.  
\end{rmk}

\begin{defn}
\label{defn:tot-geod}
 We say that $N'\subset \T$ is \emph{totally geodesic} if for any distinct $X,Y\in N'$, there exists a geodesic that has $X,Y$ in its image and lies completely in $N'$.  We say that an algebraic subvariety $N\subset \M$ is \emph{totally geodesic} if any lift $N$ under the natural projection $\T \to \M$ is totally geodesic (a \emph{lift} is an analytic irreducible component of the pre-image).  
 
\end{defn}
In the case where $\T$ or $\M$ has a single factor, the above recovers the usual definition of (Teichm\"uller) totally geodesic.  Note that a singleton set is totally geodesic, since the condition of the definition is vacuously satisfied.  

Our main theorem below gives that the boundary of a totally geodesic subvariety is totally geodesic, as well as additional information about the structure of the boundary. Similar results were obtained independently, and at the same time, in joint work of Arana-Herrera and Wright \cite{aw24}.  Their techniques are very different -- in particular, study of cylinders is central to their approach.  They also situate totally geodesic subvarieties in the world of hierarchically hyperbolic spaces, a vibrant sub-area of geometric group theory.

\begin{thm}
\label{thm:bdy-tot-geod}
    Let $N$ be a totally geodesic (complex algebraic) subvariety of $\M_{g,n}$.  Let $\mathring{\partial}N$ be an irreducible component of $\partial N\cap \Delta_{\Gamma}$, where $\Delta_\Gamma$ is some boundary stratum of $\partial \M_{g,n}$.  Then 
    $\mathring{\partial} N$ is a totally geodesic subvariety of $\Delta_{\Gamma}$.

\end{thm}

\begin{rmk}
    Even if $N$ is itself a smooth submanifold, we are not able to rule out that $\mathring{\partial} N$ is rather singular.  
\end{rmk}

\begin{example} 
\label{ex:cover}
We present an example of a totally geodesic submanifold in $\M_3$, defined in terms of a covering construction, to illustrate some of the boundary phenomena that can occur.  

Let $X$ be a genus $2$ surface, and $\widehat{X}$ a regular, double cover of $X$. See \Cref{fig:bdy}.  The deck group of $\widehat{X}\to X$ transposes the two points in each fiber. Consider the locus $N\subset \mathcal{M}_3$ consisting of such $\widehat{X}$, where $X$ ranges over $\M_2$.
\footnote{To formally define this locus, we fix a covering $\widehat S \to S$ of topological surfaces. This induces a map on \TM spaces $\T_2\to \T_3$, which is an isometric embedding.  Now let $G$ be the finite index subgroup of the mapping class group of $S$ consisting of those maps that  to this cover. One then gets a map $\T_2/G \to \M_3,$ and $N$ is the image of this map.} 
This $N$ is a totally geodesic subvariety of $\mathcal{M}_3$
(it is locally isometric to $\M_2$).

For a separating curve $\gamma$ on $X$, there is a degeneration to the boundary of $N$ obtained by simultaneously pinching the two curves $\widehat{\gamma} \in \widehat{X}$ in the preimage of $\gamma$. 
Let $\Delta_{\Gamma}$ be the boundary stratum of $\mathcal{M}_{3}$ coming from these degenerations; the product cover of $\Delta_\Gamma$ is parametrized by $\mathcal{M}_{1,2} \times \mathcal{M}_{1,1} \times \mathcal{M}_{1,1}$.  We consider $\mathring{\partial} N:= \partial N \cap \Delta_{\Gamma}$.  We claim that the product cover of $\mathring{\partial} N$ is a product $N_1\times N_2$, where $N_1\subset \M_{1,2}$ is the locus where the difference of the marked points is $2$-torsion (a dimension $1$ totally geodesic manifold),
and $N_2 \subset \M_{1,1}\times \M_{1,1}$ is the image of the diagonal embedding $\M_{1,1} \to \M_{1,1} \times \M_{1,1}$.  

To see this, we begin by observing that any element of $\mathring{\partial} N$ is obtained by the degeneration described above. Indeed, elements of $\mathring{\partial N}$ are connected, admissible covers of nodal Riemann surfaces $X$ with two genus $1$ irreducible components. %
In other words, these are nodal covers arising as limits of smooth curves, where source and target degenerate simultaneously. The construction and properness of the Hurwitz space of connected admissible covers is discussed in \cite[Section 4]{hm82}.  The deck group action at the boundary permutes the two $\M_{1,1}$ factors (i.e. the left and right Riemann surfaces in \Cref{fig:bdy}), while it acts by an involution on the $\M_{1,2}$ factor (the middle surface in the figure) given by translation by a $2$-torsion vector.  The claimed description of $\mathring{\partial N}$ then follows (the deck involution must act by a translation of order $2$, which permutes the marked points, hence there difference must be $2$-torsion.)

Observe that since $N_1,N_2$ are each totally geodesic, so is $\mathring{\partial N}$.  Our \Cref{thm:bdy-tot-geod} gives the same conclusion for any totally geodesic subvariety of $\M_{g,n}$.

\begin{figure}
    \centering
    \includegraphics[width=0.95\linewidth]{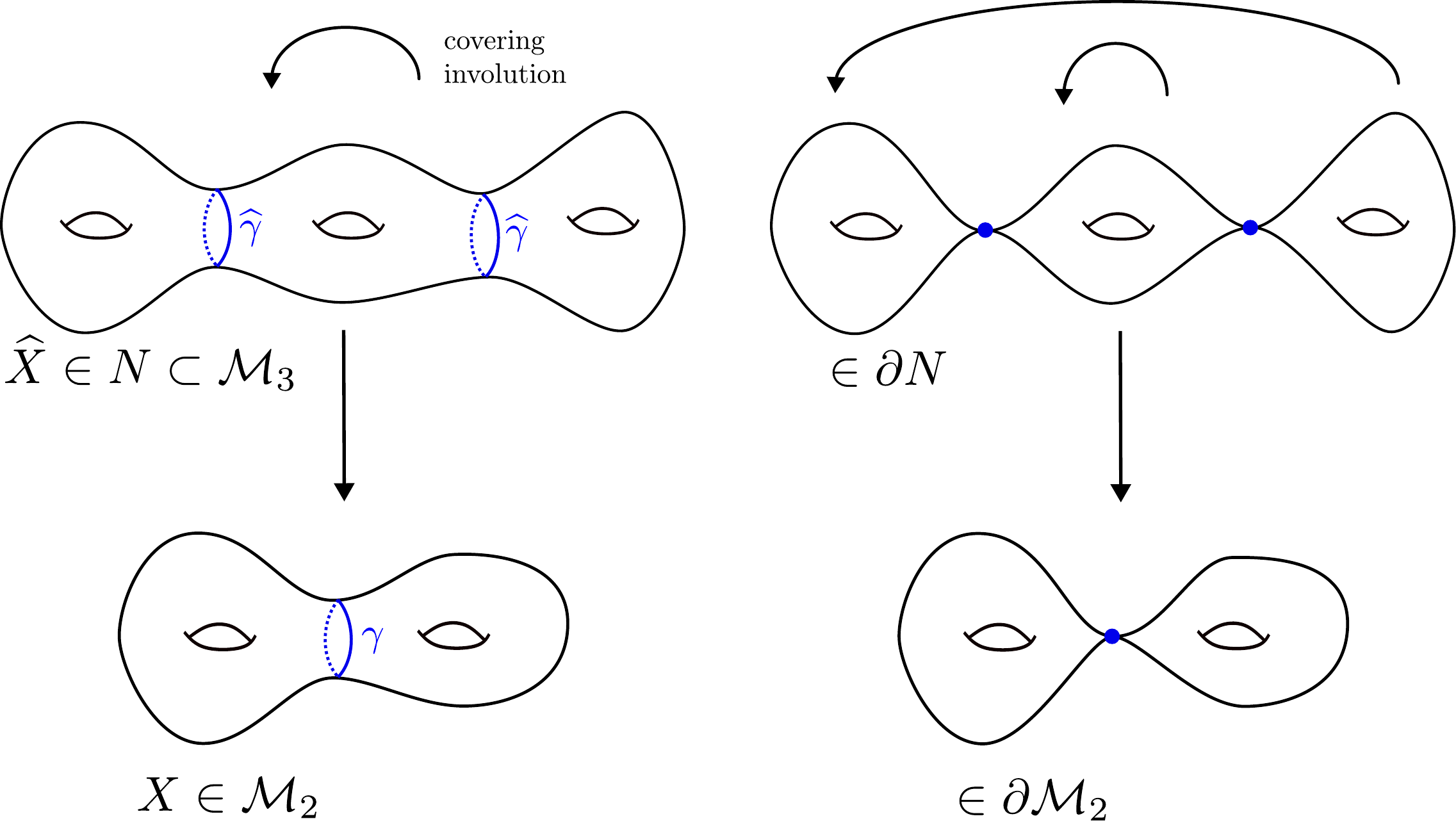}
    \caption{Boundary of totally geodesic variety coming from covering}
    \label{fig:bdy}
\end{figure}

\end{example}

\subsection{Structure of prime factors of the boundary}
\label{sec:prime}
Our method of proof yields more structural information beyond the boundary being totally geodesic.  The boundary stratum in \Cref{ex:cover} decomposes as a product of $N_1$ and $N_2$, each of which is a totally geodesic subvariety, of a single moduli space, that is diagonally embedded in some product of moduli spaces. We show that this sort of phenomenon holds in general.  

For any subset $N\subset \M_{g_1,n_1} \times \cdots \times \M_{g_k,n_k}$, we say it is \emph{prime} unless, possibly after reordering components, there is some $s$ such that $N$ can be written as a product $N'\times N''$, where $N'\subset \M_{g_1,n_1} \times \cdots \times \M_{g_s,n_s}$ and $N''\subset \M_{g_{s+1},n_{s+1}} \times \cdots \times \M_{g_k,n_k}$ .  It follows easily from the definition of totally geodesic that if $N$ is totally geodesic, then any such factors $N',N''$ must both also be totally geodesic. 

For any subset $N\subset \M$ (recall $\M$ denotes a multi-component moduli space), we say it is \textit{prime} if there exists a  of $N$ to the product cover of $\M$ which is prime in the above sense.  

By progressively decomposing, any totally geodesic subvariety of $\M$ admits a decomposition into a product of prime totally geodesic subvarieties. 

A version of the following was also proved independently, and at the same time, in joint work of Arana-Herrera and Wright \cite{aw24}: 

\begin{thm}
    \label{thm:bdy-struc} 
    Let $\mathring{\partial} N$ be as in \Cref{thm:bdy-tot-geod}. Then the product cover of $\mathring{\partial} N$ decomposes as a product of prime factors $N_j \subset \M_j$ (where $\M_j$ is a product of moduli spaces), each of which has the following property.  For each moduli space $\M_{g,n}$ that is a factor of $\M_j$, the projection map $\Phi: N_j \to \M_{g,n}$ is locally injective, in the orbifold sense.   

\end{thm}

    By $\Phi$ ``locally injective, in the orbifold sense" we mean that for any lift $\tilde N_j$ to the universal cover of $\M_j$ such that the projection $\tilde \Phi: \tilde N_j \to \T_{g,n}$ is locally injective.  This is a somewhat weaker notion than locally injective for points in the orbifold locus of moduli space.

    We will also prove in \Cref{prop:local-isom} that the projection $\Phi$ in the above theorem is a \textit{local isometry}, again in an appropriate orbifold sense.  To make sense of such a statement, we must have a notion of metric on the multi-component moduli space $\M$, rather than just a notion of geodesics (which is all we need for the rest of the paper).  It turns that our local isometry statement will hold for many of the different notions of \TM metric on a multi-componenet moduli space (these metrics are discussed in \Cref{sec:tot-geod}).

\subsection{Context} In the moduli space of Riemann surfaces equipped with the Teichm\"uller metric there is a tension between behavior shared with homogeneous spaces and behavior that is very inhomogeneous.  A particular instance of this is the character of totally geodesic submanifolds: Teichm\"uller space contains many of complex dimension $1$ (Teichm\"uller disks), while containing rather few of higher dimension.  On the level of moduli space, the dimension $1$ totally geodesic subvarieties, known as Teichm\"uller curves, have been studied extensively; many interesting examples are known, and some classification results have been proved, but the story is still incomplete.  Only in the last decade have interesting higher dimensional totally geodesic subvarieties been constructed \cite{mmw17,emmw20}.  See \cite{goujard21} for a survey of recent results.  

The study of totally geodesic subvarieties overlaps with the study $GL^+(2,\R)$-orbit closures in spaces of quadratic and Abelian differentials on Riemann surfaces. The former is in a sense a special case of the latter.  Various different boundaries of strata (see \Cref{sec:compact}) have been fruitfully used to attack the classification problem for $GL^+(2,\R)$-orbit closures. 
 Metric and analytic perspectives are potentially helpful in the case of totally geodesic subvarieties.  There have been several recent results towards classifying higher dimensional totally geodesic subvarieties \cite{benirschke24, wright20}, but a full picture still seems far off.  Our structural results on the boundary may allow inductive approaches to this problem.

\subsection{Idea of the proof of \Cref{thm:bdy-tot-geod}}

We need to produce many geodesics that lie in $\mathring{\partial} N$.  We construct a set $Q$ of quadratic differentials lying above $\mathring{\partial} N$ that is (i) $GL^+(2,\R)$-invariant, and (ii) a variety of large dimension. This is done by taking families of tangent vectors to $N$ converging to tangent vectors to $\mathring{\partial} N$, and then analyzing the quadratic differentials generating Teichm\"uller geodesics tangent to the vectors.  To get $GL^+(2,\R)$-invariance, we use the \textit{real multi-scale compactification} of quadratic differentials, and the continuous extension action of the $GL^+(2,\R)$ action to it.  From the $GL^+(2,\R)$-invariance, we get that the elements of $Q$ generate Teichm\"uller geodesics tangent to $\mathring{\partial} N$.  And from the large dimension property, the supply of these geodesics is large enough to give that $\mathring{\partial} N$ is totally geodesic.  These steps are first carried out in the case when $Q$ is \textit{prime}, where we can use a result due to Chen-Wright on constancy of ratios of areas of components.  Finally we deduce the general case by decomposing into prime factors.

\subsection{Outline of the paper}

\begin{itemize}
    \item In \Cref{sec:teich-geo}, we recall some classical facts about \TM geometry, Beltrami differentials, and quadratic differentials, and generalize to the setting of products of \TM spaces.   

    \item In \Cref{sec:tot-geod}, we define several other notions of totally geodesic, and prove that most of these turn out to be equivalent.  In particular, we prove that ``infinitesimally totally geodesic" implies totally geodesic in the sense defined above; this will be used in the proof of the main theorem.  

    \item In \Cref{sec:WYSIWYG} we recall notions to do with multi-component differentials, in particular the result that ratio of areas on prime invariant subvarieties in strata of multi-component translation surfaces are constant.  

    \item In \Cref{sec:compact}, we recall several compactifications of spaces of differentials and their properties, generalizing several results from Abelian to quadratic differentials.  

    \item In \Cref{sec:tang-quad} we consider sequences of tangent vectors converging to a tangent vector to the boundary, and study the associated quadratic differentials.  

    \item In \Cref{sec:bdy-gl2r-geod} we show, using results of the previous two sections, that there is a large dimension set of quadratic differentials generating \TM geodesics lying in our $\mathring{\partial}N$; this property is called \emph{$GL^+(2,\R)$-geodesic}.  

    \item In \Cref{sec:gl2r-geo-tot-geo} we prove \Cref{thm:bdy-tot-geod}. We use the results of \Cref{sec:bdy-gl2r-geod} and \Cref{sec:WYSIWYG} to show that $\mathring{\partial}N$ is infinitesimally totally geodesic, and then we apply the result from \Cref{sec:tot-geod} to get totally geodesic.  

    \item In \Cref{sec:bdy-struc}, we prove \Cref{thm:bdy-struc} on the structure of prime factors of the boundary of a totally geodesic subvariety.  

\end{itemize}

\subsection{Acknowledgements}

We are grateful to Curtis McMullen for helpful comments on an earlier draft of this paper.

\section{\TM geometry and differentials}
\label{sec:teich-geo}

\subsection{Beltrami and quadratic differentials}
\label{sec:diff}
We recall, in the setting of a single Teichm\"uller space, the relation between Beltrami differentials, quadratic differentials, the tangent and contangent spaces of Teichm\"uller space, and the Teichm\"uller metric.  See \cite[Section 4.6]{mcmullen-notes} for more detail.  

Given $X\in \T_{g,n}$, let $M(X)$ be the space of bounded, measurable Beltrami differentials on $X$, i.e. tensors that locally have the form $f(z) \frac{d\bar z}{dz}$, where $f$ is a bounded, measurable complex-valued function (defined up to sets of measure $0$).  The quantity $\operatorname{ess}\sup_{z\in X}|f(z)|$ is well-defined and gives a norm, denoted $\|\cdot \|$, on $M(X)$; we denote by $M_{\le 1}(X)$ the unit ball for this norm. 

Let $\calQ_X$ denote the space of quadratic differentials $q$ on $X$ that are holomorphic away from the marked points, and have at worst simple poles at the marked points.  This is a Banach space with the \emph{area norm} (or $L^1$-norm):  
\begin{align*}
    \|q\| := \int_X |q|.
\end{align*}
There is a natural pairing between $\calQ_X$ and $M(X)$:
\begin{align*}
    (q,\mu)  := \int_X q\mu. 
\end{align*}

 Now there is an analytic map $M_{\le 1}(X)\to \T_{g,n}$ that takes a Beltrami differential $\mu$ to the image of the quasiconformal map out of $X$ whose dilatation is given by $\mu$ (the existence and uniqueness of this map follows from the Measurable Riemann Mapping Theorem).  Taking the derivative induces a surjective linear map on tangent spaces $T_XM_{\le 1}(X) \cong M(X) \to T_X\T_{g,n}$.  One can prove that the kernel of this map is precisely the set 
 $$\calQ_X^{\perp} =\{\mu: (q,\mu)=0 \text{ for all } q\in \calQ_X\},$$
 known as \emph{infinitesimally trivial Beltrami differentials}, and which is a closed subspace.  Thus $T_X \T_{g,n} \cong M(X)/\calQ_X^\perp.$  We also see from this that $\calQ_X$ is naturally identified with the \emph{cotangent space} $T_X^*\T_{g,n}$. 

 \paragraph{Teichm\"uller norm.}
 \label{sec:teich-norm}
 The norm $\|\cdot\|$ we've described on $M(X)$ descends to a norm, also denoted $\|\cdot\|$, on $T_X \T_{g,n}$, the \emph{Teichm\"uller norm}; concretely we think of a tangent vector $v$ as an equivalence of Beltrami differentials $[\mu]$, and $\|v\|$ is the infimum of the $\|\cdot\|$-norm over the equivalence class.  By the duality between $L^1$ and $L^\infty$, we also have:
 \begin{align}
     \|[\mu]\| = \sup_{\{q:\|q\|=1\}} \int_X q\mu .  \label{eq:mu-char}
 \end{align}
This description of the norm is often useful, since the optimization is over the finite-dimensional space of unit norm differentials in $\calQ_X$.

 Taking the infimum of $\|\cdot \|$-length integrals along paths in $\T_{g,n}$ gives rise to the Teichm\"uller distance function.

\subsection{Bundle of quadratic differentials and $GL^+(2,\R)$-action}
\label{sec:bundle}
Given a Teichm\"uller space $\T_{g,n}$, let $\calQ \T_{g,n}$ be the bundle over $\T_{g,n}$ whose fiber over $X$ consists of meromorphic quadratic differentials on $X$ that are holomorphic away from the marked points, and have at worst simple poles at marked points (these differentials have finite area).  We let $\calQ \M_{g,n}$ be the analogous bundle over moduli space.

There is an action of $GL^+(2,\R)$ on $\calQ \M_{g,n}-\{0\}$ that plays a central role in the study of Teichm\"uller geodesics.  By cutting along saddle connections, we can represent every quadratic differential as a set of polygons in the plane, such that every side is paired up with a parallel side of equal length.  Since $GL_2(\mathbb{R})$ acts on such polygons in the plane, it acts on these surfaces.  Two subgroups will play a role in this paper:
\begin{align*}
  \left\{ g_t =
  \left(\begin{matrix}
    e^{-t} & 0 \\
    0 & e^t
  \end{matrix}\right) : t \in \mathbb{R} \right\},
 \text{  \ \    } \left\{ r_{\theta} = \left(
        \begin{matrix}
          \cos \theta & -\sin\theta \\
          \sin \theta & \cos\theta
        \end{matrix}\right) : \theta \in S^1 \right\}.
                        \label{eqn:subgps}
\end{align*}

\paragraph{Multi-component quadratic differentials.}  
\label{par:multi-comp}
Given a multi-component \TM space $\calT=\prod_i\T_i$, define $\calQ\calT:=\prod_i \calQ\calT_i$.   We define the ($L^1$) norm of $q=(q^1,\ldots,q^k) \in \calQ \calT$ by 
\[
\|q\|_1 := \sum_i \|q^i\|.
\]
 We define $\calQ\M$ and its norm similarly.

\subsection{The map $\phi$ from quadratic differentials to tangent vectors}
\label{subsec:phi}
Although quadratic differentials are naturally identified with the cotangent space, there is also an important map from quadratic differentials to the \emph{tangent} space.  

For each $X\in \calT$, we denote by $\calQ_X\calT$ the subset of $\calQ\calT$ that lies over $X$.  We define a map
\begin{align*}
\phi_X& :\calQ_X\calT \to T_X\calT, \\
\phi_X&(q)= \left(\|q^1\| \frac{\bar{q^1}}{|q^1|},\ldots, \|q^k\| \frac{\bar{q^k}}{|q^k|}\right).
\end{align*}
Note that $\frac{\bar{q^i}}{|q^i|}$ is a Beltrami differential representing the unit tangent vector to the geodesic generated by $q^i$. From the \TM existence and uniqueness theorems, it follows that $\phi_X$ is a bijection.  
We also define a map $\phi: \calQ\calT \to T\calT$ such that the restriction of $\phi$ to each fiber $\calQ_X\calT$ agrees with $\phi_X$.  

Each map $\phi_X$ is more than just a bijection:
\begin{lemma} \label{lem:homeo}
    For each $X\in \T$, the map $\phi_X: \calQ_X\T \to T_{X}\calT$ is a homeomorphism.
\end{lemma}

\paragraph{Complex geodesics.}

To prove \Cref{lem:homeo}, we will use some continuity properties of geodesics; it is fruitful to work with complex geodesics.  

Let $\Delta :=\{z\in \C: |z|<1\}$ be the unit disk. 

\begin{defn}
\label{defn:complex-geod}
 If $q\in \calQ \T_{g,n}$ is unit norm, we  define the
{\em  complex geodesic} \[
G_{\CC}(X,q): \Delta\to \T_{g,n}, \quad 
G_{\C}(X,q)(z)= X_{\phi(zq)},
\]
where $ X_{\phi(zq)}$ is obtained by solving the Beltrami equation for $\phi(zq)$.  

\end{defn}

The map $G_\C$ is connected to $\phi_X$, since the tangent vector at $z=0$ is  %
\[
G_{\C}(X,q)'(0) = \dfrac{\overline{q}}{|q|}.
\]

The next lemma states that complex geodesics depend continuously (in the topology of local uniform convergence) on the initial data.

\begin{lemma}
\label{lem:converg}
Suppose $(X_n,q_n)\in \calQ^1\calT_{g,n}$  is a sequence of unit norm quadratic differentials converging to $(X,q)\in \calQ^1\calT_{g,n}$. Then the complex geodesics $G_{\C}(X,q_n)$ converge locally uniformly to $G_{\C}(X,q)$.

\end{lemma}

\begin{proof}

Recall that the solutions of the Beltrami equation depend continuously on the parameters (in fact, analytically, see \cite{ab60}).  Using this and the convergence of $q_n$, we get point-wise convergence
\[
\lim_{n\to\infty} G_{\C}(X_n,q_n)(z) = G_{\C}(X,q)(z) 
\]
 for all $z\in \Delta.$ The functions $G_{\C}(X_n,q_n)$ are locally uniformly bounded, since each complex geodesic is an isometry for the Teichm\"uller metric. By Montel's theorem, $G_{\C}(X_n,q_n)$ converges locally uniformly to $G_{\C}(X,q)$.

\end{proof}

\begin{proof}[Proof of \Cref{lem:homeo}]
It suffices to deal with the case of a single factor $\T=\calT_{g,n}$. 
Note that in this case $\phi_X$ is norm-preserving, and hence maps the unit sphere in $\calQ_X \T$ to the unit sphere in $T_X\T$.  
Since \[
\phi_X(\lambda q)= |\lambda| \phi(q), \ \lambda\in \mathbb{C}^*,
\] it is enough to show that $\phi_X$ restricts to a homeomorphism between unit spheres. 
The area norm and the \TM norm  on $\calQ\calT$ and $T\calT$, respectively, are $C^1$, thus the unit spheres are $C^1$-submanifolds of the same dimension.  So by invariance of domain, it suffices to show that $\phi_X$ restricted to the unit sphere is continuous.

To see this, suppose we have a sequence of unit area quadratic differentials with
\[\lim_{n\to\infty }(X_n,q_n) = (X,q).
\]
The complex geodesics $G_{\C}(X_n,q_n)$ converge locally uniformly to $G_{\C}(X,q)$, by \Cref{lem:converg}. 
Thus 
\[
\phi_X(q)= G_{\C}(X,q)'(0)=  \lim_{n\to\infty} G_{\C}(X_n,q_n)'(0) =\lim_{n\to\infty} \phi(X_n,q_n). 
\]
\end{proof}

\paragraph{Alternate characterization of $\phi$.}
For later use in \Cref{sec:tang-quad}, we show the following.  
\begin{lemma}
\label{lem:cvx-uniq}
Let $\calT_{g,n}$ be a \TM space. Let $v \in \calT_X \calT_{g,n}$ non-zero and define $q:= \frac{\phi_X^{-1}(v)}{\|\phi_X^{-1}(v)\|}$. %
Then $q$ is uniquely characterized among elements of $\calQ \T_{g,n}$ by the properties:
\[
\int_{X} qv = \|v\|, \quad \|q\| = 1.
\]
\end{lemma}

\begin{proof}

    By \eqref{eq:mu-char}, we have 
    \begin{align}
     \|v\| %
                &= \sup_{\{q_0:\|q_0\| \leq 1\}} \int_{X} q_0 v. \label{eq:sup}
 \end{align}
 Now by \cite[Section 9.3, Lemma 3]{gardiner87}, the Teichm{\"u}ller norm is strictly convex.  So by the supporting hyperplane theorem, there is a unique $q_0$ achieving the supremum above, which gives the uniqueness of differential with the required two properties.

 Now let $q'=\phi_X^{-1}(v)$, so $q= \frac{q'}{\|q'\|}$.
 By the definition of $\phi_X$, we have $v=\phi_X(q') = \|q'\| \cdot \bar{q'}/q'$.  
 Then

\begin{align*}
     \int_{X} q v &= \int_{X} \frac{q'}{\|q'\|}\|q'\| \frac{\bar{q'}}{|q'|} =\int_X |q'| = \|q'\| = \|v\|,
 \end{align*}
 where for the last equality we have used that $\phi_X$ is norm-preserving (for this we need that $\T_{g,n}$ is a single \TM space rather than a product).  
 This means that $q$ has the required properties, completing the proof.  
\end{proof}

\subsection{Exponential map} 
Let $\T_i:= \T_{g_i,n_i}$ be a single Teichm\"uller space.  Given $w \in T\T_i$, let $\gamma_w:\R \to \T_i$ be the unique geodesic with $\gamma_i'(0)=w$. 
We define $E^i:T\T_i \to \T$ by $E^i(w) := \gamma_w(1)$, and 
\begin{align*}
    &E:T \T \to \T, \\
    &E(v) = E((v^1,\ldots,v^k)) := (E^1(v^1),\ldots, E^k(v^k))
\end{align*}
(we use the letter $E$ since this is an analogue of the exponential map for a Riemannian manifold).

Given $X\in \T$, we denote by $E_X$ the restriction $E|_{T_X\T}$.

\begin{lemma}
    \label{lem:exp-homeo}
    For each $X=(X_1,\ldots,X_k) \in \T$, the map $E_X:T_X \T \to \T$ is a homeomorphism.  
\end{lemma}

\begin{proof}
    By definition, $E_X = (E^1|_{T_{X_1}\T_1},\ldots, E^k|_{T_{X_k}\T_k})$.  It suffices to show that each $E_{X_i}:=E^i|_{T_{X_i}\T_i}: T_{X_i}\T_i \to \T_i$ is a homeomorphism. This map is a bijection, by the Teichm\"uller uniqueness and existence theorems.  
    
    To show continuity of $E_{X_i}$, we factor into a composition of two maps.  We can first map $T_{X_i}\T_i$ to the space of quadratic differentials on $X_i$.  This is done via $\phi^{-1}|_{T_{X_i}}$, which by \Cref{lem:homeo} is a homeomorphism $T_{X_i}\T_i \to \calQ\T_i|_{X_i}$ (recall that $\phi$ takes $q$ to $\|q\|_1\frac{\bar q}{|q|}$). We then apply the map $F:\calQ\T_i|_{X_i}\to \T_i$ defined so that $F(q)$ is given by solving the Beltrami equation for the Beltrami differential $\lambda \frac{\bar q}{|q|},$ where 
    \begin{align*}
        \lambda = \frac{e^{2\|q\|_1}-1}{e^{2\|q\|_1}+1}.
    \end{align*}
    Since the solutions of the Beltrami equation depend continuously on the data, this map $F$ is continuous.  

    The $\lambda$ above was chosen so that $E_{X_i}$ equals the composition $F \circ \phi^{-1}|_{T_{X_i}}$.  Since both of these maps are continuous, so is $E_{X_i}$.  So we have a continuous, injective map between manifolds of the same dimension; by invariance of domain it is open, and hence its inverse is also continuous.

\end{proof}

\section{Notions of totally geodesic in multi-component moduli spaces} 
\label{sec:tot-geod}
From the perspective of Teichm\"uller geometry, there are several different notions of a totally geodesic subset of a product of Teichm\"uller or moduli spaces.  We discuss the definitons, and prove relations among them; most of these turn out to be equivalent.  

\begin{defn}[$L^p$-totally geodesic]
\label{defn:Lp-tot-geod}
    For $p$ with $1 \le p \le \infty$, we endow $\T=\T_{g_1,n_1} \times \cdots \times \T_{g_k,n_k}$ with the $L^p$-Teichm\"uller  norm. That is, for $v=(v^1,\ldots,v^k) \in T\T$, we define
    \begin{align*}
        \|v\|_p := \begin{cases}
            k^{-1/p} \left( \|v^1\|^p + \cdots + \|v^k\|^p \right)^{1/p} &\text{ if } p<\infty \\
            \max_i \|v^i\| &\text{ if } p=\infty,
        \end{cases}
    \end{align*}
    where $\|\cdot\|$ denotes the Teichm\"uller norm on each $\T_{g_i,n_i}$.  (The reason we include the $k^{-1/p}$ factor is so that for a $v$ with all components of the same norm $c$, we will have $\|v\|_p=c$.) For each $p$, this gives $\T$ the structure of a \emph{Finsler manifold}.  
    Then we say that $N\subset \T$ is \emph{$L^p$-totally geodesic} if for any distinct $X,Y\in N$, any $L^p$-geodesic connecting $X$ to $Y$ lies in $N$.  We say that $N\subset \M$ is \emph{$L^p$-totally geodesic} if any lift of $N$ to $\T$ is $L^p$-totally geodesic.  By integrating the norm $\|v\|_p$ along paths, we get a path metric denoted $d_p$ on $\T$.  
\end{defn}

\begin{defn}[Infinitesimally totally geodesic]
 Let $N\subset \M$ be an irreducible subvariety.  We say that $N$ is \emph{infinitesimally totally geodesic} if there exists a dense, Zariski open subset $N'$ of $N$ such that for any tangent vector $v\in TN'$, the geodesic tangent to $v$ lies entirely in $N$.
\end{defn}

\begin{prop}
    \label{prop:tot-geod-equiv}
    Let $N\subset \M$ be a 
    algebraic subvariety.  For any $p$ with $1<p<\infty$, the following conditions are equivalent: 
    \begin{enumerate}[(1)]
        \item $N$ is totally geodesic (in the sense of \Cref{defn:tot-geod}).  
        \item $N$ is infinitesimally totally geodesic.
        \item $N$ is $L^p$-totally geodesic  
    \end{enumerate}
\end{prop}

\begin{proof}
We break into three parts:
\begin{itemize}
    \item $(2) \implies (1)$.  This is \Cref{prop:inf-geod-tot-geod}.  

    \item $(1) \implies (2)$. This is \Cref{prop:global-tangent}.  

    \item $(1) \Longleftrightarrow (3)$. Note that the geodesics in $\T$ in \Cref{defn:geod} are also $L^p$-geodesics.  On the other hand, a convexity argument, using that the Teichm\"uller metric on each $\T_{g_i,n_i}$ is uniquely geodesic, shows that $\T$ is uniquely $L^p$-geodesic (this is done in the context of general geodesic metric spaces in \cite{ks15}).  Thus the two classes of geodesics in $\T$ coincide.  It follows that the notions of totally geodesic and $L^p$-totally geodesic for subsets of $\T$ or $\M$ coincide.  
\end{itemize}
\end{proof}

\begin{rmk}
    For $p=1,\infty$, the geodesics of \Cref{defn:geod} are $L^p$-geodesics.  However $L^1$ and $L^\infty$ do not give uniquely geodesic spaces, and thus a non-trivial totally geodesic $N$ will typically fail to be $L^1$ or $L^\infty$-totally geodesic.  The $L^\infty$ metric coincides with the intrinsic Kobayashi distance on $\T$ coming from its complex structure.  This is because the Teichm\"uller and Kobayashi metrics on $\T_{g,n}$  coincide (by results of Royden), and the Kobayashi metric on a finite product of complex spaces is the supremum of the metrics on the factors (see \cite[Theorem 3.1.9]{kobayashi98}). 
\end{rmk}

\begin{prop}
\label{prop:inf-geod-tot-geod}
    Let $N\subset \M$ be a %
    subvariety that is infinitesimally totally geodesic.  Then $N$ is totally geodesic.  
\end{prop}

\begin{proof}
    We need to show that any lift $\tilde N$ of $N$ to $\T$ is totally geodesic.  Let $N'$ be the dense, Zariski open subset of $N'$ in the definition of infinitesimally totally geodesic.  We denote by $\tilde N'$ the subset of $\tilde N$ that maps to the locus $N'$.  We first show that any $X,Y\in \tilde N'$ can be connected by a geodesic in $\tilde N'$.  

    We will use the map $E_X:T_X \T \to \T$, which by \Cref{lem:exp-homeo} is a homeomorphism. 

    From the assumption that $N$ is infinitesimally totally geodesic, 
    we get that restricting $E_X$ gives a map
    $$f:=E_X|_{T_X\tilde N} : T_X\tilde N \to \tilde N.$$
    Let $S:=f^{-1}(\tilde N')$.  Note that since $N'$ is open in $N$, we have that $\tilde N'$ is open in $\tilde N$.  
    Since $f$ is continuous, it follows that $S$ is an open subset of $T_X\tilde N$, and hence a manifold of the same dimension.  Now $f|_S: S\to \tilde N'$ is a continuous and injective (by uniqueness of geodesics) map between manifolds of the same dimension, so by invariance of domain it is open.

    In particular $f(S)$ is open in $\tilde N'$.  On the other hand, $f(S) = E_X(T_X\tilde N) \cap \tilde N'$ is closed in $\tilde N'$, since $E_X(T_X\tilde N)$ is closed in $\T$ (note that $T_X\tilde N$ is closed in $T_X\T$, and $E_X$ is homeomorphism). 

    Above we have shown that $f(S)$ is both open and closed in $\tilde N'$.  Since $\tilde N$ is irreducible, $\tilde N'$ is connected.  We conclude that $f(S)=\tilde N'$, and in particular $Y$ is in the image, so it can be connected to $X$ by a geodesic.  

    It remains to deal with points in $\tilde N-\tilde N'$.  We first suppose that $X$ is in $\tilde N'$, and $Y$ any point in $\tilde N$.  Since $N'$ is dense in $N$, we can find $Y_n\in \tilde N'$ with $Y_n \to Y$.  Since $E_X$ is a homeomorphism, we have $E_X^{-1}Y_n \to E_X^{-1}Y$.  By the above, for each $n$, we have $E_X^{-1}Y_n \in T_X\tilde N$.  Since $T_X\tilde N$ is closed in $T_X\T$, we get that $E_X^{-1}Y$ is also in $T_X\tilde N.$ Applying the infinitesimal totally geodesic property that we are assuming with $E_X^{-1}Y$ then implies that the geodesic connecting $X$ to $Y$ lies in $\tilde N$.  A similar argument takes care of the case when $X,Y$ are allowed to be any points in $\tilde N.$
\end{proof}

\begin{prop}
    \label{prop:global-tangent}
    Suppose $N$ is a totally geodesic subvariety of $\M$.
    Then $N$ is infinitesimally totally geodesic.  
\end{prop}

\begin{proof}

    Let $X\in N^{reg}$, the regular (non-singular) locus, which is a dense, Zariski open subset of $N$.  
    Consider any lift $\tilde N$ of $N$ to $\T$, and let $\tilde N^{reg}\subset \tilde N$ be the subset of elements that map to $N^{reg}$ under the projection $\T\to\M$.  Let $\tilde X\in \tilde N$ be  a point projecting to $X$. %
    Recall that $E_{\tilde X}: T_{\tilde X}\T \to \T$ is a homeomorphism. It suffices to prove that $E_{\tilde X}^{-1}(\tilde N)= T_{\tilde X}\tilde N.$
    
    Towards this end, note that $E_{\tilde X}^{-1}(\tilde N)$ consists of all vectors $v$ so that the endpoint of the time $1$ geodesic segment that starts tangent to $v$ lies in $N$ (recall that our definition of geodesic allows some rescaling of the \TM metric on the target, so this geodesic need not be unit speed).  Since $\tilde N$ is totally geodesic (since $N$ is), for any such segment, the whole geodesic along it lies in $\tilde N$.  This implies 
    \begin{enumerate}[(i)]
        \item $E_{\tilde X}^{-1}(\tilde N)$ is homogeneous with respect to $\R$ scaling, and 
        \item the vector $v$ is also a tangent vector to $\tilde N$.  Thus $E_{\tilde X}^{-1}(\tilde N)\subset T_{\tilde X}\tilde N.$  
    \end{enumerate}
 
    Now consider the subset $E_{\tilde X}^{-1}(\tilde N^{reg}) \subset E_{\tilde X}^{-1}(\tilde N)$, which by (ii) above is a subset of $T_{\tilde X}\tilde N$, so we get a  restriction map
    $$E_{\tilde X}^{-1}|_{\tilde N^{reg}}: \tilde N^{reg} \to T_{\tilde X} \tilde N.$$
    This is injective (since it's a restriction of the homeomorphism $E_{\tilde X}^{-1}$), and as an injective map between manifolds of the same dimension, invariance of domain implies it is open.  Thus $E_{\tilde X}^{-1}(\tilde N^{reg})$ is an open set in $T_{\tilde X} \tilde N$, and it also clearly contains $0$.  
    
    Combining this with (i) above, we have that $E_{\tilde X}^{-1}(\tilde N)$ is homogeneous with respect to $\R$ scaling, and contains an open set in $T_{\tilde X} \tilde N$ around $0$.  So $E_{\tilde X}^{-1}(\tilde N)=T_{\tilde X}\tilde N,$ as desired.  
\end{proof}

\section{Ratios of areas in invariant subvarieties}\label{sec:WYSIWYG}
In this section, we discuss strata of multi-component surfaces and prove \Cref{prop:ratio-areas-qd}, which generalizes a result of Chen-Wright to the setting of quadratic differentials.  This will be a key tool in our proof of \Cref{thm:bdy-tot-geod}.  

\subsection{Strata of multi-component differentials} 
\label{sec:strata}

\paragraph{Strata over $\M_{g,n}$.} We begin by introducing notation for strata of differentials over $\M_{g,n}$ that deals correctly with the marked points and can handle the case of identically zero differentials.  The reader may want to first consider the case of $\M_g$ rather than $\M_{g,n}$; in this case \eqref{item:pair} and conditions \eqref{item:one-pair},\eqref{item:marked-pt} below can be ignored.  The data specifying  such a stratum will be called a \textit{divisor multi-set} $\kappa$.  This $\kappa$ either has the value $\infty$, or it is a multi-set where each element has one of two types: 
\begin{enumerate}[(i)]
    \item a single non-zero integer $a$ (these will account for zeros/poles of the differential not at marked points),
    \item a pair $(a,m)$, where $a$ is an integer, and $m$ is a positive integer (these will account for marked points). \label{item:pair}
\end{enumerate}
Given such a $\kappa$, we now define the \textit{stratum} $\calQ(\kappa)\to \M_{g,n}$. For $(X,p_1,\ldots,p_n,q)$ with $(X,p_1,\ldots,p_n,q)\in \M_{g,n}$ and $q$ a meromorphic quadratic differential on $X$, it lies in $\calQ(\mu)$ iff all the following conditions hold:
\begin{enumerate}
    \item if $\kappa=\infty$, then $q$ is identically zero (the notation is meant to suggest that the differential vanishes to infinite order), 
    \item if $a\ne 0$ appears $\ell$ times in $\kappa$, then there are exactly $\ell$ points in $X-\{p_1,\ldots,p_n\}$ where $q$ has order $a$.
    \item if $\kappa\ne \infty$, then for each $m\in \{1,\ldots,n\}$, there is exactly one pair of the form $(a,m)$ in $\kappa$.  \label{item:one-pair}
    \item for each pair $(a,m)\in \kappa$, $q$ has order $a$ at the marked point $p_m$, \label{item:marked-pt}
\end{enumerate}

We note that $\calQ\M_{g,n}$, defined in \Cref{sec:bundle}, decomposes as a a finite union of strata $\calQ(\kappa)$.

\paragraph{Multi-component strata.} Given a multi-component moduli space $\M$, we define the notion of a \textit{multi-component stratum over $\M$}, denoted  $\calQ(\mu) \to \M$.  This is most straightforward when $\M$ is a product -- given a tuple of divisor multi-sets ${{\mu}} = \left(\kappa_1,...,\kappa_k\right)$ consistent with $\M$,\footnote{in the sense that if $\M=\M_{g_1,n_1} \times \cdots \M_{g_k,n_k}$, then $\calQ(\kappa_i)$ is a stratum over $\M_{g_i,n_i}$ for each $i$ } we take $\mathcal{Q}({\mu})  := \mathcal{Q}(\kappa_1) \times \cdots \times \mathcal{Q}(\kappa_k)$.
In general, we consider the product cover $\M'$, from which $\M$ is obtained  as the quotient by a finite group $G$ (recall \Cref{foot:quotient}).  This $G$ also naturally acts on any multi-component stratum $\calQ(\mu) \to \M'$, and we take the quotient to obtain $\calQ(\mu) \to \M$.  

When the list $\mu$ excludes any $\kappa_i = \infty$, the corresponding quotient will be called a \textit{stratum of multi-component translation surfaces} (these correspond to the $\mu$ giving flat-geometric objects).

\subsection{Holonomy double covers}

We will need to generalize some known results concerning Abelian differentials to quadratic differentials.  For this we use the holonomy double cover; we begin by recalling some well-known facts about it.   

For a quadratic differential $q$ on a Riemann surface $X$, there exists a canonical double cover $\hat{X}$ and an Abelian differential $\omega$ on $\hat{X}$ so that the pullback of $q$ to $\hat{X}$ is $\omega^2$. The set of regular points includes the even zeroes of $q$, but the cover has ramification at the odd zeroes and poles. An even zero of $q$ of order $n_i$ corresponds to two zeroes of $\omega$ of orders $n_i /2$, and an odd zero of $q$ of order $n_i$ corresponds to a zero of order $n_i + 1$. The deck group of $\hat{X}$ is generated by an involution $\tau$ -- a conformal automorphism of $\hat{X}$ such that $\tau^{*}(\omega) = -\omega$. %

\paragraph{Level structures.}
For a topological surface $S$ and finite set $\Sigma$ on $S$, let $\textrm{Mod}(S,\Sigma)$ be the mapping class group relative to $\Sigma$, i.e., a quotient of the group of orientation-preserving diffeomorphisms fixing $\Sigma$ pointwise. For a stratum $\mathcal{Q}(\kappa)$ there is a finite set $\Sigma$ corresponding to zeroes and simple poles of quadratic differentials so that $\mathcal{Q}(\kappa)$ is obtained as a quotient of its universal cover, by the group $\textrm{Mod}(S,\Sigma)$. The same construction makes sense for $\mathcal{H}(\kappa)$. For $\gamma \in \textrm{Mod}(S,\Sigma)$ and $H_1(S;\mathbb{Z}/n\mathbb{Z})$ the first homology of $S$ with $\mathbb{Z}/n\mathbb{Z}$ coefficients, there is an induced endomorphism $\gamma_{*}:H_1(S;\mathbb{Z}/n\mathbb{Z}) \mapsto H_1(S;\mathbb{Z}/n\mathbb{Z})$. By a result of Serre on the construction of level-$n$ structures, 

\begin{displaymath}
\textrm{Mod}(S,\Sigma)[n]:= \left\{\gamma \in \textrm{Mod}(S,\Sigma): \gamma_{*} = id \right\}
\end{displaymath}
is torsion-free and finite index in $\textrm{Mod}(S,\Sigma)$ for $n\geq 3$. Thus, replacing $\textrm{Mod}(S,\Sigma)$ by $\textrm{Mod}(S,\Sigma)[3]$, we obtain finite manifold covers of $\mathcal{Q}(\kappa)$ and $\mathcal{H}(\kappa)$, that still admit $GL^+(2,\mathbb{R})$-actions. On these finite covers, the differentials admit no non-trivial automorphisms. By taking holonomy double covers, we have the following proposition due to Kontsevich-Zorich, upon replacing the orbifolds $\mathcal{Q}(\kappa)$ and $\mathcal{H}(\kappa)$ by finite covers that are manifolds (e.g. level 3-structures). We will denote these covers by $\widehat{\calQ(\kappa)}$ 
and $\widehat{\mathcal{H}(\kappa)}$, respectively. For products of strata $\cal{Q}(\mu)$ and $\cal{H}(\mu')$, denote by $P$ the map obtained by taking the holonomy double cover map on each factor.

\begin{prop}{\cite[Lemma 1]{Kontsevich2002ConnectedCO}}
\label{KZholonomy}
The canonical map induced on manifolds obtained from taking holonomy double covers 
\begin{displaymath}
\tilde{P}: \widehat{\calQ(\kappa)} \mapsto \widehat{\mathcal{H}(\kappa')}
\end{displaymath}
is an algebraic map, and an injective %
immersion.  The data in $\kappa'$ is obtained from $\kappa = (n_1,...,n_r,m_1,...,m_s)$ by associating to $\kappa'$ two zeroes of order $n_i /2$ for each even $n_i$ that appears in $\kappa$, one zero of order $n_i +1$ to each odd $n_i$ that appears in $\kappa$, and nothing to simple poles.
\end{prop}

\subsection{Ratio of areas in prime subvarieties}
Recall that $\M$ denotes a multi-component moduli space. We can define the notion of prime for subsets of $\cal Q \M $ -- it is completely analogous to the notion of prime for subsets of $\M$ given in \Cref{sec:prime}.  We wish to show that prime $GL^+(2,\mathbb{R})$-invariant subvarieties of (unit-area loci) of strata of multi-component translation surfaces have the following property: the ratios of areas of any two factors remains constant in the subvariety.

When we refer to a constant-area locus, of area $c$, of a subvariety of multi-component surfaces, we mean that the sum of the areas of the components is $c$. In the context of Abelian differentials, the following is due to Chen-Wright.

\begin{prop}{\cite[Corollary 7.4]{cw21}}
\label{prop:ratio-areas-ab}
The ratios of areas of factors of the unit-area locus of any irreducible component of a prime $GL^+(2,\mathbb{R})$-invariant subvariety $Q$ of a stratum of multi-component translation surfaces are constant.
\end{prop}

We now give the easy generalization of the above to the setting of quadratic differentials, using \Cref{KZholonomy}.

\begin{prop}
\label{prop:ratio-areas-qd}
The ratios of areas of factors of the constant-area $c$ locus of any irreducible component of a prime $GL^+(2,\mathbb{R})$-invariant subvariety $Q$ of a stratum of multi-component translation surfaces are constant.
\end{prop}

\begin{proof}
We assume we have lifted the constant-area $c$ locus of an irreducible component $Q$ so that $Q \subset \mathcal{Q}(\kappa_1) \times \cdots \times \mathcal{Q}(\kappa_p)$. We claim $P(Q)$ is the constant-area $2c$ locus of an irreducible component of a prime %
invariant subvariety of $\mathcal{H}(\kappa_1') \times \cdots \times \mathcal{H}(\kappa_p')$. Assume, for the sake of contradiction that $P(Q)$ is not prime. Applying Proposition $\ref{KZholonomy}$, we have by Chevalley's theorem, and the fact that the finite covering maps $\widehat{\calQ(\kappa_i)} \mapsto \calQ(\kappa_i)$  are closed, that $P(Q)$ is a $GL^+(2,\mathbb{R})$-invariant subvariety of $\mathcal{H}(\kappa_1') \times \cdots \times \mathcal{H}(\kappa_p')$, where $P = P_1 \times \cdots \times P_p$ and each $P_i$ is a holonomy double cover map on a factor. Note that $P(Q)$ is a non-trivial product of prime $GL^+(2,\mathbb{R})$-invariant subvarieties, each of which is contained in a product of a subset of the factors of $\mathcal{H}(\kappa_1') \times \cdots \times \mathcal{H}(\kappa_p')$.  Since $P$ is injective, we have $Q =  P_{\overline{1}}^{-1}(M_1) \times \cdots \times P_{\overline{r}}^{-1}(M_r)$, where each $P_{\overline{i}}$ is the product over the subset of $P_i$ that makes the pre-image of $M_i$ well-defined. We observe that every $P_{\overline{i}}^{-1}(M_i)$ must be prime, since otherwise by it can be written as a non-trivial product, and $M_i$ itself cannot be prime. This contradicts the primality of $Q$.

By the fact that strata of multi-component surfaces were defined to have finite total area and by \Cref{prop:ratio-areas-ab}, the rescalings of elements of $\mathcal{H}(\kappa_1') \times \cdots \times \mathcal{H}(\kappa_p')$ by $c/2$ are such that the ratios of areas of the factors are constant over $P(Q)$. Rescaling again by $2$, we see the ratios of areas of factors of $Q$ are also constant.
\end{proof}

\section{Compactifications of spaces of differentials}
\label{sec:compact}

We recall several different spaces giving compactifications of spaces of differentials.  These are used in several different places for various purposes later in the proof.  See \cite{dozier24} for an overview of several of these spaces and relationships between them, including flat geometric examples.  

\subsection{Quadratic Hodge bundle over $\overline{\M}_{g,n}$}
\label{sec:hodge}

Denote by $\overline{\calM}_{g,n}$ the Deligne-Mumford compactification of the moduli space of genus $g$ Riemann surfaces with $n$ (labeled) marked points $\calM_{g,n}$. %

We will describe the quadratic Hodge bundle $\calQ\overline{\calM}_{g,n}$, extending the bundle $\calQ\M_{g,n} \to \M_{g,n}$ discussed in \Cref{par:multi-comp}.  It will be a rank $3g-3+n$ (orbifold) vector bundle  over $\overline{\calM}_{g,n}$, and its elements are known as stable quadratic differentials. The projectivization of the closure of a stratum of quadratic differentials in $\mathcal{Q}\overline{\mathcal{M}}_{g,n}$ is known as the $\textit{Hodge bundle compactification}$. We provide a brief overview of these notions, and refer the reader to \cite{harris1998moduli} for a thorough treatment of stable curves, the Deligne-Mumford compactification, and the construction of the Hodge bundle as a vector bundle over the universal curve.

There exists a family $p: \mathcal{C}_{g,n} \rightarrow \calM_{g,n}$ of Riemann surfaces,  called the universal curve over $\calM_{g,n}$, which is a map of complex orbifolds. Further, there is a sheaf $\omega_{C}$ on $\overline{\mathcal{C}}_{g,n}$, 
the universal curve over the Deligne-Mumford compactification  $\overline{\calM}_{g,n}$, so that the pushforward $p_{*}\omega_{C}$ to $\overline{\calM}_{g,n}$, has the following description: $p_{*}\omega_{C}$ is a rank $g$ vector bundle, whose fibers over a stable curve $(X,p_1,...,p_n)$ is the space of global sections $H^0(X,\omega_{\pi})$ where $\omega_{\pi}$ is the relative dualizing sheaf of $X$.  The vector bundle $p_{*}\omega_{C}$ is commonly referred to as the Hodge bundle and denoted $\Omega\overline{\calM}_{g,n}$. We are interested in the \emph{quadratic Hodge bundle} $p_{*}\left(\omega^{\otimes 2}_{C}\right)$ which we denote by $\calQ\overline{\calM}_{g,n}$. This is of rank $3g-3+n$; a proof can be found in Section 4 \cite{hk14}.

An element in $\calQ\overline{\calM}_{g,n}$ is an at-worst nodal curve $X$ with a special choice of marked points $(p_1,...,p_n)$, and a quadratic differential on each irreducible component of $X$. These quadratic differentials can have poles of at most order $2$ at the nodes, %
and may have simple poles at the marked points. Poles of order 2 have a local invariant, the ``2-residue", and these residues must match at a node with a double pole, where two points on different irreducible components are identified. Away from nodes and marked points, they are holomorphic. For more on this, and the generalization of this construction to $\omega^{\otimes k}$ for $k >2$, see \cite{bainbridge2019strata}.

\subsection{Multi-scale compactification}
\label{sec:multi-scale}

In \cite{cmz22}, the authors construct a compactification of projectivized strata of $k$-differentials, the \textit{moduli space of projectivized multi-scale $k$-differentials}. This compactification was introduced in $\cite{bcggm19}$ in the Abelian case. We will work with the case $k=2$ -- compactifications of strata of quadratic differentials $\calQ(\kappa)$. We use the notation $\Xi \calQ(\kappa)$ for the (unprojectivized) space of multi-scale quadratic differentials and $\mathbb{P} \Xi \calQ(\kappa)$ for the projectivized space. Similarly, we denote the respective Abelian strata by $\Xi \cal{H}(\kappa)$ and $\mathbb{P} \Xi \cal{H}(\kappa)$.

The boundary  $\partial \mathbb{P} \Xi \calQ(\kappa)=\mathbb{P} \Xi \calQ(\kappa) - \calQ(\kappa)$ parametrizes \textit{multi-scale quadratic differentials}. These are stable curves, together with a collection of meromorphic quadratic differentials on the irreducible components, constrained by certain conditions. The boundary decomposes into a union of open boundary strata, indexed by an enhancement $\overline{\Gamma}$ of the dual graph $\Gamma$ of a stable curve. The content of the enhancement is a stratification of the irreducible components of a stable curve by \textit{levels} (corresponding to the speed of vanishing of differentials on different components), together with prescribed vanishing orders at the nodes.

In \cite[Section 7, Lemma 7.1]{cms23}, the closure of a stratum of $k$-differentials in the projectivized stratum of multi-scale $k$-differentials is described via cyclic coverings of Abelian differentials.  In particular, as the boundary strata of a projectivized multi-scale compactification of a stratum of Abelian differentials are indexed by enhanced graphs, the boundary strata of multi-scale $k$-differentials may be described by cyclic $k$-covers of the corresponding enhanced graphs. We will use this description to sketch how to adapt the following theorem to our setting of interest.

\begin{thm}{\cite[Theorem 1.2]{benirschke23}}
\label{thm:ab-level-ind}
Let $M \subset \mathcal{H}(\kappa)$ be a $\mathbb{C}$-linear subvariety. Then, the intersection of the closure $\overline{M} \subseteq \Xi\cal{H}(\kappa)$ with any boundary stratum $D_{\overline{\Gamma}}$ of the moduli space $\Xi\cal{H}(\kappa)$ of multi-scale differentials is a level-wise linear subvariety, for the natural linear structure on the boundary stratum $D_{\overline{\Gamma}} \subset \partial \Xi\cal{H}(\kappa)$. 
\end{thm}

\begin{thm}
\label{thm:quad-level-ind}
Let $Q \subset \calQ(\kappa)$ be a $\mathbb{C}$-linear subvariety. Then, the intersection of the closure $\overline{Q} \subseteq \Xi{\calQ}(\kappa)$ with any boundary stratum $D_{\overline{\Gamma}}$ of the moduli space $\Xi\cal{Q}(\kappa)$ of multi-scale quadratic differentials is a level-wise linear subvariety, for the natural linear structure on the boundary stratum $D_{\overline{\Gamma}} \subset \partial \Xi{\calQ}(\kappa)$.
\end{thm}

\begin{proof}(sketch)
Let $Q \subset \cal{Q}(\kappa)$ be a $\mathbb{C}$-linear subvariety, and consider $\tilde{P}(\widehat{Q})$, the image under $\tilde{P}$ of Proposition \ref{KZholonomy} of the lift of $Q$ to a finite manifold cover $\widehat{\cal{Q}(\kappa)}$. By Proposition $\ref{KZholonomy}$, the projection of $\tilde{P}(\widehat{Q})$ to $\cal{H}(\kappa')$ is a $\mathbb{C}$-linear subvariety of $\cal{H}(\kappa')$. Denote it by $M$. By Theorem $\ref{thm:ab-level-ind}$, $\overline{M} \cap D_{\overline{\Gamma}'}$ is a level-wise linear subvariety of the boundary stratum $D_{\overline{\Gamma}'}$. Apply the map $d_{\pi}$ in \cite[Lemma 7.2]{cms23} to see that $\overline{M} \cap D_{\overline{\Gamma}'}$ is mapped by $d_{\pi}$ to a level-wise linear subvariety of a boundary stratum $D_{\overline{\Gamma}}$ in $\partial \Xi{\calQ}(\kappa)$. Furthermore, $d_{\pi}\left(\overline{M} \cap D_{\overline{\Gamma}'}\right) = \overline{Q} \cap D_{\overline{\Gamma}}$.
\end{proof}

\subsection{Real multi-scale space}
\label{sec:real-multi}
Finally, in $\cite{bcggm19}$, a real, oriented blowup of $\Xi \cal{H}(\kappa)$ is constructed, which we denote by $\widehat \Xi \cal{H}(\kappa)$. This has the advantage that the $GL^+(2,\mathbb{R})$-action extends continuously from $\mathcal{H}(\kappa)$ to $\widehat \Xi \cal{H}(\kappa)$ \cite[Theorem 15.1]{bcggm19}. In a similar way, a real blowup of $\Xi \cal{Q}(\kappa)$, denoted by $\widehat \Xi \cal{Q}(\kappa)$ may be constructed that admits a continuous $GL^+(2,\mathbb{R})$-action.

There are natural forgetful maps between the various spaces we have introduced:
\[
\widehat \Xi \calQ(\kappa) \to \Xi \calQ(\kappa) \to \calQ \Mgn \to \Mgn.
\]
Note that the underlying Riemann surface of an element of $\Xi \calQ(\kappa)$ may have some extra $\P^1$ components relative $\calQ \Mgn$.  This is because these surfaces have marked points at all the zeros of the quadratic differential.  Yet there is still a natural map $\Xi \calQ(\kappa) \to \calQ \Mgn$, since there is a canonical way to contract these extra $\P^1$ components (which become unstable when the zeros of the differential are forgotten).

\section{Tangent vectors and quadratic differentials}
\label{sec:tang-quad}
The goal of this section is to use a family of tangent vectors in $T\M_{g,n}$ converging to a vector in $T\partial \Mgn$ to produce converging quadratic differentials.  This is \Cref{prop:tangent}, which needs several technical assumptions.  \Cref{lem:hol-family} gives the existence of families of tangent vectors satisfying these assumptions and that are tangent to any fixed subvariety of $\M_{g,n}$.

\subsection{Preliminaries}
\paragraph{Coordinates for $\Mgn$ at the boundary.}
 We recall the description of $\Mgn$ at the boundary in terms of complex-analytic plumbing coordinates. The set-up we require below can be found, in more detail, in Section 3 of \cite{wolpert13}.

Let $X$ be the normalization of a point in a boundary stratum $\Delta_\Gamma$ of $\Mgn$; we will think of this as a multi-component Riemann surface with pairwise identified points $\{a_i,b_i\}_{i=1,\ldots,M}$.  We choose local coordinate maps $w_i,\zeta_i$ near these points -- we can assume that $w_i(a_i)=0$, $\zeta_i(b_i)=0$, and that the unit disc is contained in the image of both maps.  %

Since the boundary stratum $\Delta_\Gamma$ is itself a multi-component moduli space, its tangent space at $X$ is product of tangent spaces to individual moduli spaces.  Hence, by the discussion in \Cref{sec:diff}, we can choose Beltrami differential $\nu_j$, $i =1,\ldots,N$, giving a basis for $T_X\Mgn$. The $\nu_i$ can be chosen so that their supports are small open sets disjoint from the domains of the coordinates $w_k,\zeta_k$.
For $s\in \C^N$ small enough, let $X_s$ be the deformation of $X$ obtained by solving the Beltrami equations for $\nu(s) = \sum_i s_i\nu_i$ in the open set where it is supported. 
Then let $X_{s,t}$ be the Riemann surface obtained by plumbing $X_s$ with parameters $t=(t_1,\ldots,t_M)$.  That is, for each $i$ with $t_i\ne0$, we remove the open balls $w_i^{-1}(\{z:|z|<t_i\})$ and $\zeta_i^{-1}(\{z:|z|<t_i\})$ from $X_s$, and then glue together all pairs of points for which $w_i\zeta_i=t_i$.

Let $U = S\times \Delta^M$ be the resulting coordinate chart.  
The  vector fields \[
(\partial_{s_1},\ldots,  \partial_{s_N}, \partial_{t_1},\ldots, \partial_{t_M})
\] form a frame for the tangent bundle $T\Mgn|_{U}.$

\paragraph{Formula for $\partial_{t_i}$.}
Away from the boundary, a Beltrami differential representing the infinitesimal variation $\partial_{t_i}$ %
was calculated explicitly in \cite[Formula 2.2]{masur76} to be
\[
\frac{1}{2 t_i\log |t_i| } \frac{w_i}{\bar w_i} \frac{d\bar w_i}{dw_i}
\]
 in a neighborhood of  the node corresponding to $t_i$, and identically zero everywhere else.
 In particular, away from the boundary
 \begin{align}
\|\partial_{t_i}\| \leq \frac{1}{-2 t_i\log\ |t_i|}  \label{eq:HK}.  
 \end{align}

\begin{rmk}
Here is a sketch of the proof of the bound \eqref{eq:HK}.
We will find quasiconformal maps between the surfaces corresponding to plumbing parameters $t$ and $t'$, compute their Beltrami differentials, and then compute the derivative of this with respect to $t'$ at $t'=t$. %

Instead of directly finding the  maps between surfaces, it's convenient to conformally map the plumbing annulus (of outer radius 1, inner radius $t$) to a Euclidean cylinder of circumference $2\pi$, via the map $w\mapsto \log w$. Between these cylinders, there are nice affine quasiconformal maps, given by stretching.  The map going from the cylinder of height $h$ to $h'$ has Beltrami form of norm approximately $(h-h')/(h+h')$.  Hence the derivative with respect to $h'$ at $h'=h$ is around $-1/(2h)$.  We now move back to the plumbing annuli (taking the map between surfaces to be constant outside of plumbing regions).  Using the Chain Rule, we get that the Beltrami form here has norm
$$\approx \frac{1}{-2t \log t}.$$
\end{rmk}

\paragraph{Quadratic differentials of bounded area.}
   A potential issue later on is that there are converging families of quadratic differentials in the Hodge bundle with cylinders of very quickly growing moduli that lie on a part of the surface that is converging to a zero component of the limit $q$.  The contribution to area from this cylinder can be arbitrarily large.  However, the content of the next lemma is that this phenomenon does not occur for quadratic differentials lying over \textit{holomorphic} families of Riemann surfaces.  
\begin{lemma}
    \label{lem:area-cts}
    Let $f:\Delta \to \overline{\M}_{g,n}$ be a holomorphic family, with $f(0)\in \partial\Mgn$ and $f(z)\in \M_{g,n}$ for $z\ne 0$.  Let $q\in \calQ \Mgn$ be a finite area quadratic differential on $f(0)$, let $z_k\in \Delta$ satisfy $\lim_{k\to\infty} z_k\ = 0$, and let $q_k\in \calQ \Mgn$ lie over $f(z_k)$ such that $\lim_{k\to\infty} q_k = q$.  Then the areas satisfy
    \[
    \lim_{k\to\infty} \|q_k\| = \|q\|_1.  
    \]
\end{lemma}

\begin{proof}
    We choose a plumbing neighborhood $U$ of $f(0)$ in $\Mgn$.  Since the desired statement is about limits as $z\to 0$, we can assume that $f(\Delta)\subset U$.   There are plumbing parameters $t_i': U \to \C$ that specify how the nodes of $f(0)$ are smoothed.  A point $z\in U$ is in the boundary $\partial \Mgn$ iff $t_i'(z)=0$ for all $i$.  Each $t_i'$ is a holomorphic function on $U$, and hence $t_i' \circ f: \Delta \to \C$ is also holomorphic.  By considering the power series at $0$ of this composition, and using that $t_i'\circ f(z)=0$ iff $z=0$ (by the hypothesis on where the values of $f$ lie), we see that all the $t_i'$ are polynomially bounded above and below in terms of $z$.  That is, for each $i$ there is a positive integer $d_i$ and $C_i>0$ such that for all $z\in \Delta$,
    \[    
        \frac{1}{C_i} |z|^{d_i}  \le |t_i'\circ f(z)|  \le C_i |z|^{d_i}.
    \]
    
    Now it suffices to show the desired convergence along a subsequence of the $q_k$.  By passing to such a subsequence, we can assume the $q_k$ all lie in a fixed stratum $\calQ(\kappa)$, and that they converge to a limit $q^{MS}$ in the multi-scale space $\Xi \calQ(\kappa)$
    \footnote{\label{foot:subseq} To see this, take a subsequence such that the projective classes $[q_k]$ converge to some point $[q^{MS}]\in \P\Xi \calQ(\kappa)$, using the compactness of this space.  Since the multi-scale space contains more refined information than the Hodge bundle, the top level of $[q^{MS}]$ must equal the projective class of the non-zero components of $q$, and so we can choose the representative $q^{MS}$ to agree with $q$ on these components.  Since convergence on lower levels is the same in $\P\Xi \calQ(\kappa)$ as in $\Xi\calQ(\kappa)$, we see that in fact $q_k \to q^{MS}$ in $\Xi \calQ(\kappa)$.} 
    introduced in \Cref{sec:multi-scale}
     All $q_k$ for $k$ sufficiently large lie in a \emph{perturbed period coordinate} neighborhood $V$ of $q^{MS}$.  The top part of $q^{MS}$ agrees with the top part of $q$, and in particular, it also has area $\|q\|_1$.  
    
    We recall some additional basic facts about the multi-scale space, and its perturbed period coordinates.  The space is defined in \cite{bcggm19}, examples from the flat perspective are presented in \cite{dozier24}, and some of the flat geometric implications are discussed in \cite{dozier23}.  The boundary point $q^{MS}$ is specified by a \emph{multi-scale differential}, which consists of a meromorphic quadratic differential $\eta$ on each component of the underlying stable Riemann surface (together with combinatorial data: a level structure on the dual graph, and prong-matching data -- the latter will not be important for us). To smooth $q^{MS}$, for each $\eta$ at level $m$, we glue the rescaled differential $t_1^{a_1} \cdots t_k^{a_m} \eta$ to the analogously rescaled differentials on the other components (we may also need to add in a modification differential to account for certain residues, but this is small and won't affect our area estimates below). Here the $t_j$ are \textit{scaling parameters}, which constitute some of the coordinates on $V$.  Each $a_i$ is a positive integer associated to level $i$ that can be computed in terms of orders of certain poles of $q^{MS}$; the exact formula will not be important here.
    
    To smooth each horizontal node, a different plumbing procedure is used, governed by a \textit{horizontal node parameter} $t$; this produces a flat cylinder, whose ratio of height to circumference is on the order of $\log|t|^{-1}$.   
    
    If $\eta$ does not have horizontal nodes, then the area it contributes to a smooth surface in $V$ is 
    \begin{align}
    O\left( t_1^{a_1} \cdots t_m^{a_m}\right). \label{eq:no-horiz} 
    \end{align}
    For each horizontal node on $\eta$ (which either connects a component to itself, or to some other component at the same level), there is an area contribution of 
    \begin{align}
        O\left( (t_1^{a_1} \cdots t_m^{a_m} )^2 \log |t|^{-1} \right), \label{eq:horiz}
    \end{align}
    since the circumference of the cylinder is $O(t_1^{a_1} \cdots t_m^{a_m})$.

    Now we consider the relation between the Riemann surface plumbing parameters $t_i'$ and the multi-scale plumbing coordinates $t_j,t$.  Each horizontal node parameter $t$ has size on the order of the $t_i'$ for that node, since both are on the order of the modulus of the largest conformal annulus homotopic to the corresponding vanishing cycle. %
    From our discussion at the beginning of the proof, we then see that $|t|$ is polynomially bounded in terms of $|z|$.  

    The size of the plumbing parameter $t_i'$ corresponding to a node that connects two components at levels $m$ and $\ell$, with $m>\ell$ (with respect to the level structure that is part of the data of $q^{MS}$) is of the same order as some (possibly fractional) power of the product  $t_{m-1}^{a_{m-1}} \cdots t_\ell^{a_\ell}$.  This, as previously, is because both quantities are on the order of the modulus of the largest conformal annulus homotopic to the vanishing cycle of the node.   %
    
    Using the fact that the level graph is connected, we can go in reverse, expressing the order of each $t_j$ in terms of a product of rational, possibly negative, powers of the $t_i'$. 
    (By exploring the graph, for each level $k$, one expresses $t_{-1}^{a_{-1}} \cdots t_k^{a_k}$ in terms of a product of some rational, possibly negative, powers of the $t_i'$.  One can then divide to express each $t_j$.) 
    From the discussion at the beginning of the proof, this means that $t_j(q_k)$ is on the order of $|z_k|^r$, where $r$ is some rational, possibly negative, number. We know that $\lim_{k\to\infty} t_j(q_k)=0$, since $q_k \to q^{MS}$ and $t_j(q^{MS})=0$.  Since $\lim_{k\to\infty} z_k$ is also $0$, we conclude that the exponent $r$ must in fact be positive.  

    Now returning back to our estimate of area of $q_k$, we see that the contribution \eqref{eq:no-horiz} has order bounded above by a power of $|z_k|$, while for \eqref{eq:horiz} it is some power of $|z_k| \log|z_k|^{-1}$.  The point is that both of these tend to $0$ as $k\to \infty$.  Hence the area contributed to $q_k$ by any $\eta$ not at top level tends to $0$ as $k\to\infty$.  From the top level components, the area contribution tends to the area of the top part of $q^{MS}$.  It follows that the area $\|q_k\|$ tends to the area of the top part of $q^{MS}$, which 
    equals $\|q\|_1$.  
    
\end{proof}

\subsection{Converging vectors and quadratic differentials}

Our goal now is to start with a family of tangent vectors to $\M_{g,n}$ that converge to a tangent vector $v$ to the boundary $\partial \M_{g,n}$, and then get a corresponding converging family of quadratic differentials.  We will need a technical assumption on the family, and the resulting quadratic differential will be pinned down only up to some non-negative scalars.  

\paragraph{Norms on products.}
We will have to compare norms of tangent vectors and quadratic differentials.
Recall that in \Cref{par:multi-comp}, we defined the $\|\cdot \|_1$ norm on quadratic differentials as the sum of areas.  For tangent vectors to a single moduli space $\calM_{g,n}$ we defined the \TM norm in \Cref{sec:teich-norm}.  And in \Cref{defn:Lp-tot-geod}
we defined $\|\cdot\|_\infty$ on $T\calT$; we extend in the obvious way to $T\calM$.  

We will use the Hodge bundle $\calQ\Mgn$ introduced in \Cref{sec:hodge}.  In fact, we will primarily work with the subspace $\calQ_{<\infty}\Mgn$ of \textit{finite area} differentials, i.e. those with no double poles.  
\begin{defn}
\label{defn:top-level}
    For $v\in T_X\M$, where $\M$ is any multi-component moduli space, we define the \emph{top level tangent vector} $\top(v)$ to be the element of $T_X\M$ that is zero on each component $X_i$ of $X$ where $\|v^i\|< \|v\|_\infty$, and agrees with $v$ on all other components. 
\end{defn}

\begin{prop}\label{prop:tangent}
Let $\M_{g,n}$ be a single moduli space and  $f:\Delta \to T\overline{\M}_{g,n}$ 
be a holomorphic family with
\[
f(z) = (X({s(z),t(z)}),v(z)) = \sum_{i=1}^N c_i(z) \partial_{s_i} + \sum_{j=1}^M d_j(z) \partial_{t_j}\in T\overline{\M_{g,n}}
\] such that 
\begin{enumerate}[1.]
\item $(X(z):=X(s(z),t(z)),v(z))$ is in $T\M_{g,n}$ when $z\ne 0$. \label{item:open}
\item $(X(0),v(0))\in  T\partial \M_{g,n}$ and $v(0)$ is not identically zero on every component, \label{item:bdy}
\item $d_j(z) = O(t_j(z))$ for $j=1,\ldots m$, \label{item:tranverse-decay}

\end{enumerate} 

For $z\in \Delta^*$, let 
\[
q(z) := \frac{\phi^{-1}(v(z))}{\|\phi^{-1}(v(z))\|}.  
\] 
Then any sequence from the family $q(z)$ with $z\to 0$ has an accumulation point $q \in \calQ_{<\infty}\Mgn$, and any such $q$ agrees with $\phi^{-1}(\top(v))$, up to rescaling each factor $i$ by some non-negative real $c_i$. Furthermore, $q$ is not identically zero.

\end{prop}

\begin{rmk}
    At present, we do not know if the conclusion of \Cref{prop:tangent} holds if we remove condition \ref{item:tranverse-decay}.
\end{rmk}

\begin{proof}
Here is the idea of the proof.  The norm on tangent vectors comes from a sup norm on Beltrami differentials, and thus at the boundary should limit to a sup of the norms of the tangent vectors on the component Riemann surfaces.  This is not literally true, because the Beltrami differentials efficiently representing the tangent vectors might have significant contributions from the thin parts of the Riemann surfaces, and these are hard to control in the limit.  However, in our setting condition \ref{item:tranverse-decay} means that this thin part contribution is not substantial (this is quite technical and makes the proof complicated).  Then we use this, together with a characterization of the map $\phi$ in terms of pairings between quadratic differentials and Beltrami differentials, to control the corresponding quadratic differentials.

Let $\mathcal{Q}_{[0,1]}\Mgn \subset \mathcal{Q} \Mgn$ be the locus of differentials with area at most $1$.  Let $(X(z_n), v(z_n))$ be a sequence converging to $(X(0), v(0))$ and $q$ be an accumulation point of $q(z_n)$ in $\calQ_{[0,1]}\Mgn$, which exists since this space is compact (see \cite[p.1238]{Mcmullen2013NavigatingMS} for the case of Abelian differentials; the quadratic differential case is similar). To prove the desired result, it suffices to assume that $q_n\to q$. %
The differential $q=(q^1,\ldots,q^k)$ can be decomposed into components and
\[
\|q\|_1 = \operatorname{area}(q) = \sum_{i=1}^k \operatorname{area}(q^i)\leq 1.
\]
For brevity, write \[
 X_n = X(z_n), \ X= X(0), \ v_n = v(z_n), \ v= v(0).
\]

\begin{claim}
\label{claim:pairing-cts}
Suppose $\xi_n\to \xi$, with $\xi_n$ an element of $\calQ \M_{g,n}$ lying above $X_n$ such that $\|\xi_n\|$ is bounded independent of $n$, and $\xi \in \partial \calQ_{<\infty} \Mgn$ lying above $X$.  Then
\[
\lim_{n\to\infty} \int_{X_n} \xi_n v_n = \int_{X} \xi v.
\]
\end{claim}

\begin{proofclaim}
By assumption  $v(z)$  converges to a vector $v=v(0)= (v^1(0),\ldots,v^n(0))$ at the multi-component Riemann surface $X = (X_1,\ldots, X_n)$.
The tangent vector $v(z)$ can be expressed in terms of the frame $(\partial_s,\partial_t)$ as 
\[
\begin{gathered}
v(z) = u(z)+R(z), \\
u(z):=\sum_{i=1}^N c_{i}(z)\partial_{s_i},\\
R(z):= \sum_{j=1}^M d_{j}(z)\partial_{t_j}.
\end{gathered}
\]

Condition \ref{item:tranverse-decay} implies 
\[
 d_{j}(z) = O(t_j(z)), 
\]
and in particular $d_j(0)=0$, so $u(0)=v(0)$.  
For $z\neq 0$, the tangent vectors $u$ and $R$ can be represented by Beltrami differentials on the smooth surface $X(z)$.

Now we claim that $\partial_{s_i}$ on $X_n$ can be represented by a Beltrami differential $\nu_{i,n}$ that converges to $\nu_i$, in the sense that there is an exhaustion of $X-\{\text{nodes}\}$ by compact sets $K_n$ and quasiconformal maps $f_n:K_n \to X_n$ with dilatation tending to $0$, such that the pullbacks $f_n^*(\nu_{i,n})$ converge to $\nu_i$ uniformly.  The quasiconformal maps we use are furnished from the local solutions of the Beltrami equation in the description of the surfaces $X_{s,0}$ composed with the conformal identification, away from the nodes, of $X_{s,0}$ with $X_{s,t}$ given by the plumbing construction.

The $\nu_{i,n}$ are defined by first finding representatives of $\partial_{s_i}$ on deformed surfaces $X_s=X_{s,0}$ within the boundary using the the composition rule for quasiconformal maps (see \cite[p. 425]{wolpert13}).
From this description one also sees that these Beltramis converge to $\nu_i$ in the required sense.  This representative of $\partial_{s_i}$ on $X_s$ is supported away from the nodes, and then the identification given by plumbing allows us to transport this Beltrami to define a Beltrami on any $X_{s,t}$; this is how we define $\nu_{i,n}$ on $X_n$.  Since the relevant Beltramis are supported away from the nodes, the deformations specified by them commute with the plumbing deformation.  Hence $\nu_{i,n}$ does in fact represent $\partial_{s_i}$.  We also see that $\nu_{i,n}$ converges to $\nu_i$ in the required sense.  

From the above, and using the ``quasiconformal topology" interpretation of $\xi_n\to \xi$ (see \cite[Proposition 3.5]{bcggm19}) it follows that 
\begin{align}
    \lim_{n\to\infty} \int_{X_n} \xi_n \partial_{s_i} = \lim_{n\to\infty} \int_{X_n} \xi_n \nu_{i,n} = \int_X \xi \nu_i = \int_X \xi \partial_{s_i}. \label{eq:estimate-s}  
\end{align}

To deal with the $\partial_{t_j}$ components of $v$, we note, using the assumption that $\|\xi_n\|$ is bounded and  the estimates \ref{item:tranverse-decay} for $d_j(z_n)$ and \eqref{eq:HK} for $\partial_{t_j}$, that 
\begin{align}
    \left| \int_{X_n} \xi_n d_j(z_n) \partial_{t_j}\right| &\le  \|\xi_n\| \cdot |d_j(z_n)| \cdot \|\partial_{t_j}\| =O\left(t_j(z_n) \cdot \frac{1}{-t_j(z_n) \log|t_j(z_n)|} \right)\\
    & = O\left( \frac{1}{- \log|t_j(z_n)|} \right) 
        \to_{n\to\infty} 0. \label{eq:estimate-t}
\end{align}

We now combine the estimates \eqref{eq:estimate-s} and \eqref{eq:estimate-t} to get 
\begin{align*}
    \int_{X_n} \xi_n v_n & = \sum_{i=1}^N c_i(z_n) \int_{X_n} \xi_n \partial_{s_i} + \sum_{j=1}^M \int_{X_n} \xi_n d_j(z_n) \partial_{t_j} \\
    & \to_{n\to\infty} \sum_{i=1}^N c_i(0) \int_{X} \xi \partial_{s_i} + 0= \int_X \xi u(0)= \int_X \xi v(0),
\end{align*}
and we are done.  
\end{proofclaim}

Now breaking into components, and using the characterization \eqref{eq:mu-char} of $\|v^i\|$ in terms of the pairing with quadratic differentials, gives
\[
\begin{split}
\int_{X} qv
&=\sum_{i=1}^n \int_{X^i} q^iv^i  \leq  \sum_{i=1}^n \sup_{\{q_0^i: \|q_0^i\|=\|q^i\|\}} \int_{X^i} q_0^i v^i \\
&= \sum_{i=1}^n \|q^i\| \|v^i\|\leq \|q\|_1 \|v\|_{\infty}\\
&\leq \|v\|_{\infty}.
\end{split}
\]

Now \Cref{lem:cvx-uniq} gives $\|v_n\| = \int_{X_n} q(z_n)v_n $. 
 Using this, \Cref{claim:pairing-cts} applied with $q_n\to q$, and some of the above, we get the following chain of inequalities which will use after proving the next Claim:
\begin{align}
    \lim_{n\to\infty} \|v_n\| = \sum_{i=1}^n \int_{X^i} q^iv^i \le  \sum_{i=1}^n \|q^i\| \|v^i\|\leq \|q\|_1 \|v\|_{\infty} \leq \|v\|_{\infty}. \label{eq:ineq-chain}
\end{align}

We remark here that the inequality $\lim_{n\to\infty} \|v_n\| \le \|v\|_\infty$ could fail if we did not assume condition \ref{item:tranverse-decay}.  To see this, choose a smooth path $\gamma:[0,1] \to \overline \M_{g,n}$ with $\gamma([0,1)) \subset \M_{g,n}$, $\gamma(1) \in\partial \M_{g,n}$, and such that $v := d\gamma/dt|_{t=1}$ is in $T\partial \M_{g,n}$.    We get a family $\gamma'(t)$ of tangent vectors, and taking values at $t_n$ tending to $1$ gives a sequence $v_n\to v$ satisfying conditions \ref{item:open} and \ref{item:bdy}.  The length $|\gamma|$ in the Teichm\"uller metric is $\infty$, since the boundary is infinitely far away. This implies %
$$\limsup \|v_n\|=\infty > \|v\|_\infty.$$

\begin{claim} 
\label{claim:norm-lim}
We have
\[
\liminf_{n\to\infty} \|v_n\| \geq \|v\|_{\infty}.
\]
\end{claim}

\begin{proofclaim}
We start by finding a $\xi\in \calQ \Mgn$ lying above $X$ with $\|\xi\|_1 =1$ and $\int_X \xi v = \|v\|_\infty$.  Using that $\calQ \Mgn$ is a vector bundle, we can find $\xi_n\in \calQ \Mgn$ lying above $X_n$ with $\xi_n \to \xi$.  Now convergence in the Hodge bundle does not in general imply convergence of areas, but since our $\xi_n$ lie over a holomorphic family, we can apply \Cref{lem:area-cts} to conclude that $\lim_{n\to\infty} \|\xi_n\| =1$, and in particular these quantities are bounded.  We now apply \Cref{claim:pairing-cts}, to get that
\[
\lim_{n\to\infty} \int_{X_n} \xi_n v_n = \int_X \xi v = \|v\|_\infty.
\]

Now $\|v_n\| \ge \frac{\int_{X_n} \xi_n v_n}{\|\xi_n\|}$, and hence  
\[
\liminf_{n\to\infty } \|v_n\| \ge \liminf_{n\to\infty} \frac{\int_{X_n} \xi_n v_n }{\|\xi_n\|}
=\frac{\|v\|_\infty}{1},
\]
as desired. 
\end{proofclaim}

By the above Claim, all the inequalities in \eqref{eq:ineq-chain} must be equalities, and so $\|q\|_1 = 1$ and 
$$\sum_{i=1}^n \| q^i\| \|v^i\| =\|v\|_\infty.$$
Note that $\sum_{i=1}^n \|q^i\| =\|q\|_1= 1$, which means that the left-hand side of the above is a weighted average of the $\|v^i\|$.  Now $\|v\|_\infty = \max_i \|v^i\|$, so the only way the above inequality can hold is for $q^i=0$ whenever $\|v^i\|< \|v\|_\infty.$

From the inequalities in \eqref{eq:ineq-chain} actually being equalities, we also get that
\[
\int_{X^i} q^iv^i = \| q^i\| \|v^i\|, \text{\ for all } i.
\]
 In particular for all $i$ such that $\|v^i\|=\|v\|_\infty$, the above together with \Cref{lem:cvx-uniq} implies that $q^i$ equals $\phi^{-1}(v^i)$ up to rescaling by some non-negative real number.  
 Since $\|q\|_1=1$, we also see that $q$ is not identically zero.

\end{proof}

\subsection{Tangent families to subvarieties}
To be able to use \Cref{prop:tangent}, we will need to produce families satisfying its hypotheses.  The next result allows us to do that for any subvariety.  

\begin{lemma}
\label{lem:hol-family}
Let $N\subseteq \calM_{g,n}$ be an algebraic subvariety, $\mathring{\partial} N$ be an irreducible component of $\partial N \cap \Delta_\Gamma$.  Let $X\in \mathring{\partial} N$ be a generic point. Then there exists a holomorphic family $f:\Delta\to \overline{N}$ with $f(0)=X$ satisfying the following property:
for any $v\in T_X(\mathring{\partial} N)$ that is not identically zero on every component,  there exists a lift $\tilde{f}:\Delta\to T\overline{N}$ of $f$ %
 with $\tilde{f}(0)= (X,v)$ and satisfying the assumptions $\eqref{item:open}-\eqref{item:tranverse-decay}$ in \Cref{prop:tangent}.

\end{lemma}

\begin{proof}
Consider the pair $(\overline{N},\mathring{\partial} N)$ and choose a  resolution \[
r:(Z,D)\to (\overline{N},\mathring{\partial} N)
\]such that $Z$ is smooth, $D$ is a normal crossing divisor.
Choose an open subset $U$ of the regular (non-singular) locus $(\mathring{\partial} N)^{reg}$ such that $U':=r^{-1}(U)\to U$ is smooth. In particular given $v\in T\mathring{\partial} N|_U = TU$ there exists a lift $(X',v')\in TU'$ with $r_*(X',v')=v.$
Choose local coordinates 
 $x=(x_0,x_1,\ldots,x_k)$  on $Z$ such that in a neighborhood  of a point in $U'$ we have
\[
D= \{x_0 =0 \}, \quad v' = \sum_{i=1}^k c_i\partial_{x_i}.
\]

Given a point $ x'=(x_1,\ldots, x_n)\in U'$ consider the path 
\[
f(z) = (z,x_1,\ldots,x_n).
\]
For generic $x'\in U'$, the path $f$ lies generically in the preimage of the regular part of $N$.

Let
\begin{align*}
 h&:\Delta \to TZ, \\
 z& \mapsto \left(\alpha(z)\sum_{i=1}^k c_i\partial_{x_i}\right).
 \end{align*}
The image of $h$ is generically contained in  \[
  TZ|_{r^{-1}(N^{reg})}.
 \]

 We now post-compose $h$ with the derivative map
 \[
 r_*:TZ\to T\calM_{g,n},
 \]

go get $\tilde{f}:=r_*\circ h:\Delta \to T\calM_{g,n}$.

Properties \eqref{item:open} and \eqref{item:bdy} are satisfied by construction but \eqref{item:tranverse-decay} needs justification.
The divisor $\{t_i(x) =0 \}$ is a multiple of $\{x_0=0\}$ and hence we can write
$t_j(x) = x_0^kh_j(x)$ for some $k$ and $h_i$ an analytic function with $h_i(0)\neq 0$.  In particular
\[
\dfrac{\partial t_j}{\partial x_i} = x_0^k \dfrac{\partial h_j}{\partial x_i} = O(t_j) \text{ if $i\neq 0$}.
\]
We start by computing $r_*(x,v')$. Recall that in local coordinates the pushforward is represented by the Jacobian matrix $(\tfrac{\partial s}{\partial x},\tfrac{\partial t}{\partial x})^t$.
We have 
\[
r_{*}(x,v') = \sum_{i=1}^m\tilde{c}_i(x)\partial s_i + \sum_{j=1}^n \left(\sum_{i=1}^k \dfrac{\partial{t_j} }{\partial x_i} c_i\right)\partial t_j
=
 \sum_{i=1}^m\tilde{c}_i(x)\partial s_i + \sum_{j=1}^n O(t_j)\partial t_j.
\]
Here $\tilde{c}_i$ are holomorphic functions with $\tilde{c}_i(0,x_1,\ldots,x_n)= c_i.$
Thus $d_j(x) = O(t_j)$. The same is true after composing with $h$.

\end{proof}

\section{Boundary is $GL^+(2,\R)$-geodesic}
\label{sec:bdy-gl2r-geod}

The main goal of this section is \Cref{prop:bdy-gl2r-geod}, which gives that $\mathring{\partial} N$ is $GL^+(2,\R)$-geodesic (defined below).  This means that there is a large dimension set of quadratic differentials generating Teichm\"uller geodesics lying in $\mathring{\partial}N$.  

We will first discuss several general notions concerning a multi-component moduli space $\M$ and a multi-component stratum $\calQ(\mu)$ over $\M$.  We denote the projection forgetting the differential by $p:\calQ(\mu) \to \M$.  

\subsection{$GL^+(2,\R)$-action} We define a $GL^+(2,\R)$ action on any $\calQ(\mu)$ by taking the usual $GL^+(2,\R)$ action (defined in \Cref{sec:teich-geo}) on components where the differential is not identically zero, and the identity on identically zero factors.  Though this gives an action on the whole quadratic Hodge bundle $Q\overline{\M_{g,n}}$, it is \emph{not} continuous; the issue is with differentials converging to zero on some component of the boundary.  Nevertheless, this action will be useful to us.  

\subsection{$GL^+(2,\R)$-geodesic}
\label{sec:gl2r-geod}
The below definition captures having many quadratic differentials whose $GL^+(2,\R)$ orbits lie in our subvariety.   

\begin{defn}
\label{defn:gl2r-geod}
Let $N\subseteq \calM$ be an irreducible, analytic subvariety.  Suppose that there exists a stratum $\calQ(\mu)\to \M$, $\mu=(\kappa_1,\ldots,\kappa_m)$, and an irreducible algebraic subvariety $Q\subseteq \calQ(\mu)$, the \emph{witness}, such that:
\begin{enumerate}[(i)]
 \item \label{item:gl2-inv} $Q$ is $GL^+(2,\R)$-invariant,
\item \label{item:proj} $\overline{p(Q)} =  N$, %
\item \label{item:dimQ} $\dim Q \ge 2\dim N$, %
\item \label{item:non-vary} $Q$ does not vary on any component $i$ for which $\kappa_i=\infty$.
\end{enumerate}
Then $N$ is called {\em $\GL^+(2,\RR)$-geodesic}.

\end{defn}

\subsection{The space $QN$}

The connection between totally geodesic subvarieties $N$ in a single moduli space and $GL^+(2,\R)$-invariant subvarieties comes from the fact that Teichm\"uller geodesics are generated by quadratic differentials.  Hence the set of quadratic differentials generating Teichm\"uller geodesics lying in $N$ will be a large and interesting set when $N$ is totally geodesic.  

For the remainder of this section, $N\subset \M_{g,n}$ will be a totally geodesic subvariety, $\Delta_\Gamma$ a boundary stratum of $\barMgn$, and $\mathring{\partial} N$ an irreducible component of $\partial N \cap \Delta_\Gamma$.

\begin{defn}
    Define 
    \[
    QN: = \{q\in \calQ\M_{g,n}: p(g_t q) \in N \text{ for all } t\in \R\}. 
    \]
\end{defn}

We will denote by $\partial QN$ the boundary of $QN$ in the Hodge bundle $\calQ \Mgn$.

\begin{prop}
\label{prop:QN-var}
    The set $QN$ is a $GL^+(2,\R)$-invariant subvariety of $\calQ\M_{g,n}$.  
\end{prop}

\begin{proof}
    The set is closed, since $N$ is. It is manifestly $g_t$-invariant, for any $t\in \R$.  Any $q\in QN$ generates a Teichm\"uller geodesic lying in $N$, and in particular its tangent vector $v$ at $p(q)$ lies in $TN$.  Now for any $\theta \in S^1$, $r_\theta q$ generates a Teichm\"uller geodesic $\gamma$ which is tangent to $r_\theta v$, and this vector also lies in $TN$ since $N$ is a complex variety.  By \Cref{prop:global-tangent}, $N$ is infinitesimally totally geodesic, so if $q$ lies over a generic point of $N$, then $\gamma$ is in fact contained in $N$, so $r_\theta q \in QN$.  Since $N$ is totally geodesic, $p(QN)=N$, and so such generic $q$ are certainly dense in $QN$.  By continuity, it follows that $r_\theta q\in QN$ for all $q\in QN$ and $\theta \in S^1$.  Since $g_t,r_\theta$ (together with real scaling) generate $GL^+(2,\R)$, we get that $QN$ is $GL^+(2,\R)$-invariant.  

    Finally we can apply \cite{emm15} and \cite{filip16} to conclude that $QN$ is a subvariety.  

\end{proof}

\subsection{Limits of quadratic differentials in $QN$}
We would like to show that for any $v\in T\partial N$, the associated quadratic differential $\phi^{-1}(v)$ lies in the boundary $\partial QN$.  We are unable to directly show this, but instead we prove \Cref{prop:qd-seq}, which produces an element of $\partial QN$ that agrees with $\phi^{-1}(v)$ up to certain rescalings by positive reals.  To this end, we use the results of \Cref{sec:tang-quad} to prove several preliminary lemmas, which give elements of $\partial QN$ related to $\phi^{-1}(v)$, though only on certain components.  

Recall the $\top(v)$ notation from \Cref{defn:top-level}.

\begin{lemma}
    \label{lem:conv-top}
    Let $X\in \mathring{\partial} N$ be a generic point. Then there exists a family $\{X(z)\}_{z\in \Delta}$ of elements of $N$ with the following property.  For any non-zero $v\in T_X\mathring{\partial} N$ and sequence $X_n:=X(z_n)$ with $z_n\to 0$, there exists $q_n$ quadratic differentials on these $X_n$ that converge to some $q\in \calQ_{<\infty}\Mgn$, not identically zero, along a subsequence, where 
    \begin{enumerate}[(i)]
        \item $q_n\in QN$, and 
        \item $q$ is a quadratic differential that agrees with $\phi^{-1}(\top(v))$ up to rescaling each factor $i$ by some non-negative real number $c_i$.
    \end{enumerate}  
\end{lemma}

\begin{proof}
    By \Cref{lem:hol-family}, there exists a holomorphic family of surfaces $\{X(z)\}_{z\in \Delta}$ such that for any $v\in T_X\mathring{\partial} N$, there exists lift of this family to tangent vectors $\{(X(z),v(z))\}_{z\in \Delta}\in T\overline N$ with $v(0)=v$, and satisfying assumptions \eqref{item:open}-\eqref{item:tranverse-decay} of \Cref{prop:tangent}.  %
    Applying \Cref{prop:tangent} to this then gives a sequence $\{v_n\}$ such that $q_n:=\phi^{-1}(v_n)/\|\phi^{-1}(v_n)\|$ satisfy condition (ii). %
    Since $N$ is totally geodesic, by \Cref{prop:global-tangent} it is infinitesimally totally geodesic, so $q_n\in QN$ for all $n$.  Thus condition (i) holds as well.  %
\end{proof}

\paragraph{Balanced vectors in the boundary.}

A tangent vector $v\in T_XN$ is called \emph{balanced} if $\| v^i\| = \|v^j \|$ for all $i,j$ with $v^i\neq 0,v^j\neq 0$.
Given $v\in T_XN$ we define the \emph{balanced vector} $b(v)$ by 
\begin{align*}
    b(v)^i = 
    \begin{cases}
        \frac{v^i}{\|v^i\|} &\text{\ \ if \ } v^i\neq 0 \\
        0 &\text{\ \ otherwise.} 
    \end{cases}
\end{align*}

\begin{lemma}
    \label{lem:levelwise-indep} 
     If $q\in \partial QN$ then $b(\phi(q))\in T\partial N$.  
\end{lemma}

 \begin{proof}
      Choose $q_k \in QN$ with $\lim_{k\to\infty} q_k = q$.  By passing to a subsequence, we can assume that all $q_k$ lie in a single stratum $\calQ(\mu)$, and (passing to a further subsequence) that they converge in the multi-scale space $\Xi \calQ(\mu)$ to some boundary point $q^{MS}$ (see \Cref{foot:subseq}). The top level of $q^{MS}$, after removing unstable components, coincides with the non-zero components of $q$.  Let $Y$ be the closure of $QN$ in the boundary stratum containing $q^{MS}$.  By \Cref{thm:quad-level-ind}, Benirschke's result in our quadratic setting, $Y$ is a level-wise linear subvariety cut out by linear equations in period coordinates with real coefficients.  This implies that, for each $t$, the result of applying $g_t$ to the top level differentials of $q^{MS}$, while not changing the lower levels, remains in $Y$. Under the map from Multi-scale to Hodge, this path projects to a path lying in $\partial QN$.  %
      The tangent vector to this path is precisely $b(\phi(q))$.   So $b(\phi(q))\in T \partial N$.  
 \end{proof}

\paragraph{Decomposition of set of components.} Given a tangent vector $v=(v^1,\ldots,v^k) \in T\partial N$, we specify a decomposition $\mathcal S (v)$ of $\{1,\ldots,k\}$ into disjoint sets $S_1\sqcup \cdots \sqcup S_n$, by the property that for $i\in S_\ell, j\in S_{\ell'}$, we have $\|v^i\| > \|v^j\|$ if $\ell<\ell'$, and $\|v^i\| = \|v^j\|$ if $\ell=\ell'$.  We also stipulate that the $S_i$ are all non-empty, except if $v$ has no zero components, in which case $S_1,\ldots S_{n-1}$ are non-empty while $S_n$ is the empty set.  (That is, each part of the decomposition corresponds to factors of the same norm, the parts are in descending order of norm, and the last part consists of the factors with norm $0$.)

For any $\ell$, we define a vector $L_\ell(v)\in T\partial N$ by 
    \begin{align*}
        L_\ell(v)^i = \begin{cases}
            v^i \text{ \ if } i \in S_\ell \\
            0 \text{ \ otherwise.}  
        \end{cases} 
    \end{align*} 
In other words, $L_\ell(v)$ is the vector given by zeroing out on all components that don't lie in the $\ell$ part.

\begin{lemma}
    \label{lem:top-level-tan} 
     Let $X\in \mathring{\partial} N$ be generic, and $v\in T_X(\mathring{\partial} N)$.  
    Then $\top(v)$ lies in $T_X(\mathring{\partial} N)$.  
\end{lemma}

\begin{proof}
    The proof we give is similar to that of \Cref{claim:qd-bdy-top} below.  This follows the same pattern as that, but this setting is somewhat simpler since here the statement does not involve quadratic differentials (they are still used in this proof).  
    
    We argue by induction on the number $\eta$ of non-zero components of $v$.  The base case is when $\eta=0$, i.e. $v=0$, for which the statement is trivial.  
    
    For the inductive step, we assume that we have proved the result when the number of non-zero components is $1,\ldots,\eta-1$, for some $\eta\ge 1$, and that $v$ is a vector with exactly $\eta$ components that are non-zero.  We begin by applying \Cref{lem:conv-top}, giving a non-zero $q\in \partial QN$ that agrees with $\phi^{-1}(\top (v))$ rescaled by a non-negative real number $c_i$ on the $i$th component.  

    We then apply \Cref{lem:levelwise-indep} with this $q$, giving that $b(\phi(q))\in T\partial N$.  By the relationship between $q$ and $\phi^{-1}(\top (v))$ discussed above, we can find a suitable scalar multiple $w:=c \cdot b(\phi(q)) \ne 0$ that is zero on all components for which it does not agree with $v$. Since this tangent vector $w$ comes from tangent vectors to the components of $X$, we have $w\in T(\Delta_\Gamma)$, and by using the genericity assumption on $X$, we can ensure that $w\in T_X(\mathring{\partial} N)$ (we can exclude $X$ in the locus where $\mathring{\partial} N$ intersects other irreducible components of $\Delta_\Gamma \cap \partial N$).  
      
    Now by linearity of $T_X(\mathring{\partial} N)$, it follows that $v-w\in T_X(\mathring{\partial} N)$.  Now $v-w$ has strictly fewer components that are non-zero than $v$ does, i.e. fewer than $\eta$,
    so by the inductive hypothesis, we get that $\top(v-w)\in T_X(\mathring{\partial} N)$.  Using linearity again, we get that $w+\top(v-w) \in T_X(\mathring{\partial} N)$; this vector is exactly $\top(v)$, so we are done.  
\end{proof}

The following is the main result of this subsection.  

\begin{prop}
\label{prop:qd-seq}
     Let $X\in \mathring{\partial} N$ be generic, and $v\in T_X(\mathring{\partial} N)$.  Then there exists $q\in \partial QN$ 
     such that for each $i$, $q^i=c_i \phi^{-1}(v^i)$ for some $c_i\in \R_{>0}$.  Furthermore, we can take $\sum_{j\in S_\ell} c_j = 1$ for all $\ell$.  
\end{prop}

\begin{proof}
    We begin by applying \Cref{lem:conv-top} with $X$, which gives a family $\{X(z)\}_{z\in\Delta}$.  

    We will first consider the case  $v=\top(v)$.  We will argue by induction, which will involve passing to subsequences and adding quadratic differentials.  For this reason we need to start with an arbitrary sequence of surfaces from our family, and find quadratic differentials along a subsequence of this.    
    
    \begin{claim}
        \label{claim:qd-bdy-top}
        Let $X\in \mathring{\partial} N$ be generic, and $v\in T_X(\mathring{\partial} N)$ with $v=\top(v)$, and suppose we are given a sequence $X_n$ from the family, i.e. $X_n=X(z_n)$ for some $z_n\to 0$.  Then there exist $q_n$ quadratic differentials on $X_n$ with $q_n \in QN$ such that $q_n$ converges, along some subsequence, to a $q\in \calQ_{<\infty}\Mgn$, where for each $i$, $q^i=c_i \phi^{-1}(v^i)$ for some $c_i\in \R_{>0}$.  Furthermore, we can take $\sum_{j\in S_1} c_j = 1$.
    \end{claim}

    \begin{proofclaim}
        We argue by induction on the number $\eta$ of non-zero components of $v$.  The base case is when $\eta=0$, i.e. $v=0$, for which the statement follows by taking all $q_n$ to be identically zero.      
        
        For the inductive step, we assume that we have proved the result when the number of non-zero components is less than $\eta$, for some $\eta\ge 1$, and that $v$ is a vector with exactly $\eta$ components that are non-zero.  By \Cref{lem:conv-top}, there exists $q_n$ quadratic differentials on $X_n$ with $q_n\in QN$ converging, along some subsequence $n_j$, to a $q\ne 0$  with the following property: $q$ agrees with $\phi^{-1}(v)$ rescaled by a non-negative real number $c_i$ on the $i$th component.  

        By \Cref{lem:levelwise-indep} (levelwise independence) applied with this $q$, we get a non-zero $w\in T_X(\mathring{\partial} N)$ that is zero on all components for which it does not agree with $v$.  By linearity of $T_X(\mathring{\partial} N)$, it follows that $v-w\in T_X(\mathring{\partial} N)$.  Now $v-w$ has strictly fewer components that are non-zero than $v$ does, i.e. fewer than $\eta$.  Applying the inductive hypothesis with this $v-w$, gives that there exists $q'_{n_j}$ quadratic differentials on $X_{n_j}$ with $q'_{n_j}\in QN$ converging, along some subsequence, to a $q'\ne 0$  with the following property: $q'$ agrees with $\phi^{-1}(v-w)$ rescaled by a \emph{positive} real number $c_i$ on the $i$th component.  Note that $\phi^{-1}(v-w)$ is non-zero exactly on the components where $v$ is non-zero and $q$ is \emph{zero}.  Thus $q+q'$ equals $\phi^{-1}(v)$, up to \emph{positively} rescaling each factor.
        
        To finish the proof we consider the sequence $q_{n_j}+q'_{n_j}$ of differentials on $X_{n_j}$.   By continuity of addition of differentials parameterized by the quadratic Hodge bundle, $q_{n_j}+q'_{n_j}$ converges, along some subsequence, to $q+q'$.  Now since $q_{n_j},q'_{n_j}\in QN$, by \cite[Proposition 2.1]{benirschke24}, $q_{n_j}+q'_{n_j}\in QN$.  Thus $q_{n_j}+q'_{n_j}$ is a sequence with the desired properties, once we rescale each by a single factor to ensure that $\sum_{j\in S_1} c_j =1$ (for $n$ not belonging to the subsequence $n_j$, we can choose $q_n$ arbitrarily).

    \end{proofclaim}

    We now relax the restriction that $v=\top(v)$.  
    
    \begin{claim}
        \label{claim:qd-bdy}
        Let $X\in \mathring{\partial} N$ be generic, and $v\in T_X(\mathring{\partial} N)$, and suppose we are given a sequence $X_n$ from the family, i.e. $X_n=X(z_n)$ for some $z_n\to 0$.  
        Then there exist $q_n$ quadratic differentials on $X_n$ with $q_n\in QN$ such that $q_n$ converges, along some subsequence, to $q \in \calQ_{<\infty}\Mgn$, where for each $i$, $q^i=c_i \phi^{-1}(v_i)$ for some $c_i\in \R_{>0}$.  Furthermore, we can take $\sum_{j\in S_\ell} c_j = 1$ for all $\ell$.  
    \end{claim}
    
    \begin{proofclaim}
        Any such $v$ can be written as
        $$v = L_1(v) + \cdots + L_n(v).$$
        
        We argue by induction on $n$.  For the base case $n=1$, the statement is just \Cref{claim:qd-bdy-top}.  
        
        For the inductive step, assume that for some $n\ge 2$ we have proved the result for $1,\ldots,n-1$, and suppose $v$ is a vector as above.  By \Cref{lem:top-level-tan}, each $L_1(v) \in T_X(\mathring{\partial} N)$.  Applying \Cref{claim:qd-bdy-top} with $L_1(v)$ gives $q_n$ quadratic differentials on $X_n$ with $q_n \in QN$ such that $q_n$ converges, along some subsequence $n_j$, to a $q$ that agrees with $\phi^{-1}(L_1(v))$ up to positively rescaling factors, and with the desired condition on sums of scaling factors.    

        Now we apply the inductive hypothesis with vector $w:=L_2(v) + \cdots + L_n(v)$, and sequence $\{X_{n_j}\}$.  This yields $q'_{n_j}$ quadratic differentials on $X_{n_j}$ with $q'_{n_j} \in QN$ such that $q'_{n_j}$ converges, along some subsequence $n_j$, to a $q'$ that agrees with $\phi^{-1}(w)$ up to positively rescaling factors, and with the desired condition on sums of scaling factors.  Then $q+q'$ equals $\phi^{-1}(v)$, up to positively rescaling each factor, and with the desired condition on sums of scaling factors. 

        Finally, we consider the sequence $q_{n_j}+q'_{n_j}$ of differentials on $X_{n_j}$. By continuity of addition of differentials parameterized by the quadratic Hodge bundle, $q_{n_j}+q'_{n_j}$ converges, along some subsequence, to $q+q'$.  Now since $q_{n_j},q'_{n_j}\in QN$, by \cite[Proposition 2.1]{benirschke24}, $q_{n_j}+q'_{n_j}\in QN$.  Thus $q_{n_j}+q'_{n_j}$ is a sequence with the desired properties (for $n$ not belonging to the subsequence $n_j$, we can choose $q_n$ arbitrarily).  
    \end{proofclaim}

The lemma follows immediately from the previous Claim.
    
\end{proof}

\subsection{Boundary in Hodge bundle}
In the previous subsection, we produced lots of elements of $\partial QN$.  Our goal now is to deduce properties of this whole space, in order to use it to produce geodesics in the direction of every tangent vector to $\mathring{\partial} N$.  Actually, rather than all of $\partial QN$, we work with a large subset of it, $\mathring{\partial} QN'$ defined below. This set has nicer properties, that we prove below.  This set will then be used in \Cref{prop:bdy-gl2r-geod} to produce a witness that $\mathring{\partial} N$ is $GL^+(2,\R)$-geodesic.  

\paragraph{The sets $\mathring{\partial} QN$ and $\mathring{\partial} QN'$.}

Recall that $\mathring{\partial} N$ is an irreducible component of $\partial N\cap \Delta_{\Gamma}$ (where $\Delta_{\Gamma}$ is a boundary stratum of $\partial \M_{g,n}$). 

\begin{defn}
 Let $\mathring{\partial} QN \subset \partial QN$ consist of those elements lying over $\mathring{\partial} N$. 
\end{defn} 

Since, by \Cref{prop:QN-var}, $QN$ is an algebraic variety, so is $\partial QN$, and hence $\mathring{\partial} QN$ is also.  

We might try to study this variety $\mathring{\partial} QN$ directly.  It may well contain differentials that vanish on some components of the underlying stable Riemann surface.  A priori, there could be components on which $\mathring{\partial} N$ does not vary, and it is natural to expect that elements of $\mathring{\partial} QN$ would vanish on those components.  On the other hand, it could be that an element of $\mathring{\partial} QN$ is ``accidentally" zero on some component, in the sense that $\mathring{\partial} N$ does vary on that component.  These differentials present an issue, and in particular their presence means that it is not clear that $\mathring{\partial} QN$ is $GL^+(2,\R)$ invariant.  Hence, we want to exclude such differentials, which suggests the definition below.  

For each factor of the product cover of $\Delta_{\Gamma}$, the product cover of $\mathring{\partial} N$ can be either \textit{varying} on that factor (i.e. the projection map to that factor is non-constant), or \textit{non-varying}.  

\begin{defn}
    Let $\mathring{\partial} QN'$ be the subset of $\mathring{\partial} QN$ consisting of those $q$ whose product cover is zero exactly on the factors where $\mathring{\partial} N$ is varying.  
\end{defn}

\begin{lemma}
    \label{lem:tan-glob-inv}
    The subset $\mathring{\partial} QN'$ is:
    \begin{enumerate}[(1)]
        \item \label{item:alg} an algebraic subvariety of $\mathring{\partial} QN$, and
        \item \label{item:tang} equals the subset consisting of those $q$ whose product cover is zero exactly on those factors where every tangent vector to the product cover of $\mathring{\partial} N$ is zero.  
    \end{enumerate}
\end{lemma}

From here on we will assume for notational simplicity that $\Delta_{\Gamma}$ is already a product (having to take product covers does not change anything substantially).  
\begin{proof}
     For \eqref{item:alg}, note that $\mathring{\partial} QN'$ is the intersection of $\mathring{\partial} QN$ with a union of strata of the boundary, so it too is an algebraic variety. %

     For \eqref{item:tang}, we need to show that $\mathring{\partial} N$ is non-varying exactly on those factors for which every tangent vector to it is zero.  One direction is immediate: if $\mathring{\partial} N$ is non-varying on some factor, then certainly every tangent vector to it is zero on that factor.  For the other direction, suppose every tangent vector to $\mathring{\partial} N$ is zero on factor $i$.  This implies that the projection map $\rho_i$ to the $i$th factor must be locally constant on (the smooth locus of) $\mathring{\partial} N$. %
     Since $\mathring{\partial} N$ is irreducible, this implies that $\rho_i$ is constant on $\mathring{\partial} N$, hence we have non-varying on factor $i$. 
\end{proof}

\paragraph{$GL^+(2,\R)$-invariance. }
Recall that there are several properties that we need a witness to $\mathring{\partial}N$ being $GL^+(2,\R)$-geodesic to satisfy.  The set $\mathring{\partial} QN'$ will not be the witness, since the witness needs to lie in a single stratum.  Rather, we will eventually intersect with some stratum to produce the witness.  

The next lemma, which gives $GL^+(2,\R)$-invariance, is a key step in the proof of our main theorem.  The idea is to use the real multi-scale space, which admits a continuous $GL^+(2,\R)$-action, to get that the relevant boundary is $GL^+(2,\R)$-invariant.  This is inspired by the analogous continuity property of the $GL^+(2,\R)$ action on the WYSIWYG space \cite[p. 934]{mw2017}, used by Mirzakhani-Wright to great effect.  The reason we use the real multi-scale space rather than WYSWIYG is that we need some control on components on which the differentials vanish.  

\begin{lemma}
    \label{lem:gl2-inv}
    The subset $\mathring{\partial} QN'$ of the Hodge bundle is invariant under the $GL^+(2,\R)$-action.
\end{lemma}

\begin{proof}
    Consider some $q\in \mathring{\partial} QN'$ lying over a generic point of $\mathring{\partial} N$.  There exists some sequence $\{q_n\}$ in $QN$ converging to $q$.  By passing to a subsequence if necessary, we can assume that all $q_n$ lie in a fixed stratum $\calQ(\mu)$.  
    
    We consider the real multi-scale space $\Xi \calQ(\mu)$ associated to $\calQ(\mu)$ introduced in \Cref{sec:real-multi}.  By passing to a subsequence if necessary, we can assume that the $q_n$ converge in the real multi-scale space  to a real multi-scale differential $q^R$ (see \Cref{foot:subseq}).
    In \Cref{sec:real-multi} several maps were discussed; we will use $\pi: \Xi\calQ(\mu)\to \calQ \Mgn$.  Since $\pi$ is continuous, we have $\pi(q^R)=q$. 

    Now the real multi-scale space admits a \emph{continuous} $GL^+(2,\R)$-action, extending the action on $\calQ(\mu)$.  So, for any $t\in \R$, we get that $g_tq_n \to g_tq^R$ in the real multi-scale space, as $n\to \infty$.  It follows that $g_tq_n \to \pi(g_tq^R)$ in the Hodge bundle.  Since $g_tq_n$ is also in $QN$ (by its very definition), we get that $\pi(g_tq^R)\in \partial QN$.  

    Note that the $g_t$ action on the boundary of the real multi-scale space preserves the boundary stratum %
    of the underlying Riemann surfaces.  From this, and since locally near $q$ the varieties $\mathring{\partial N}$ and $\partial N$ agree (since $q$ was assumed to lie over a generic point of $\mathring{\partial} N$), we get that for $t$ of small absolute value, the projection of $g_tq^R$ to $\overline\M_{g,n}$ lies in $\mathring{\partial} N$.  But the components of $g_tq^R$ corresponding to vanishing differentials are exactly the same as those for $q$ (up to contractions of some $\P^1$ components).  And since $\mathring{\partial} N$ does not vary on these components (since $q\in \mathring{\partial} QN'$), we get that $g_tq^R$ agrees with $q$ in terms of the Riemann surfaces (up to contractions of some $\P^1$ components) on these vanishing differential components.  On the other hand, from the way the actions are defined, $g_tq$ and $g_tq^R$ agree for the components where the differential does not vanish.  Hence $\pi(g_tq^R) = g_tq$, and so $g_tq \in \partial QN$.  Since the components on which the differentials vanish are exactly the same for $q$ and $g_tq$, and we've already seen that the projection of $g_tq^R$ to $\overline\M_{g,n}$ (and hence of $g_tq$) lies in $\mathring{\partial} N$, we in fact get that $g_tq\in \mathring{\partial} QN'$.  
    
    So we have shown that for generic $q$, and small $t$ (depending on $q$), $g_tq\in \mathring{\partial} QN'$. \Cref{lem:local-inv} allows us to relax the generic and small conditions -- we get for any $q$ and $t\in \R$ that $g_tq \in \mathring{\partial} QN'$.  Since $\mathring{\partial} QN'$ is a complex algebraic variety, it is invariant under $r_\theta$, for any $\theta \in S^1$.  Since these $g_t,r_\theta$ generate $GL^+(2,\R)$, we get the desired result.  
\end{proof}

\begin{lemma}
    \label{lem:local-inv}  
    Suppose $b_t$ is a continuous action of $\R$ on a real-analytic variety $M$, such that the orbit maps are real-analytic (in the sense that for any open $U\subset M$, $x\in U$, and real-analytic function $f:U\to \R$, the function $t\mapsto f(b_tx)$ is real-analytic on the open set $\{t: b_tx \in U\}$).  Let $X\subset M$ be a real-analytic subvariety with the following invariance property: for a generic $x\in X$, there exists some $t_0=t_0(x)>0$ such that $b_tx\in X$ for all $t$ with $|t|<t_0$.  Then $X$ is $b_t$-invariant.  
\end{lemma}

\begin{proof}
    We first claim that for generic $x\in X$, we have that $b_tx\in X$ for all $t\in \R$.  %
    It suffices to show that the set 
    $$S:=\{t_0\in \R^+: b_tx \in X \text{ for all } t \text{ with } |t|<t_0\},$$
    is both closed and open in $\R^+$ (it is non-empty by hypothesis).  
    
    Closedness is immediate from the definition.  
    
    For openness, let $t_0\in S$, and consider a neighborhood $V$ of $b_{t_0}x$ in $M$ such that $X\cap V$ is cut out in $V$ by real-analytic functions defined on all of $V$.  There is some $\epsilon$ such that $b_tx\in V$ for all $t\in (t_0-\epsilon, t_0+\epsilon)$.  
    Let $f:V\to \R$ be an analytic function vanishing on $X\cap V$, and consider the function $g:(t_0-\epsilon, t_0+\epsilon)\to \R$ given by $t\mapsto f(b_tx)$.  This is also real-analytic, by the assumption on analyticity  of orbit maps.  It vanishes for all $t\in (t_0-\epsilon,t_0)$, since for these $b_tx\in X\cap V$.  But since it's real-analytic, $g$ must then vanish on all of $(t_0-\epsilon, t_0+\epsilon)$. 
    So for $t$ in this interval, $b_t x$ is in the vanishing locus of $f$.  
    
    Since this argument holds for any $f$, we get that $b_tx \in X$ for any $t\in (t_0-\epsilon, t_0+\epsilon)$.  Together with an analogous argument starting with $b_{-t_0}x$, we conclude that $(t_0-\epsilon', t_0+\epsilon')\subset S$, for $\epsilon'$ the minimum of the $\epsilon$ produced for $b_{t_0}x$ and $b_{-t_0}x$.  So we have shown openness, and thus our first claim.  

    Finally to get the desired result for \textit{all} $x\in X$ from the above, we use density of generic points, the continuity of the action, and the fact that $X$ is closed in $M$.  

\end{proof}

\paragraph{Dimension bound and the map $B_\mathcal S$.}
The next property that we address towards using $\mathring{\partial} QN'$ to produce a witness that $\mathring{\partial}N$ is $GL^+(2,\R)$-geodesic is the dimension lower bound \eqref{item:dimQ}.  

If we knew that $\phi^{-1}(v) \in \partial QN$ for each $v\in T\partial N$, this would be easy, since $T\partial N$ has the right dimension, and $\phi$ is a homeomorphism.  We will use \Cref{prop:qd-seq}, proved above, which is a related statement, but somewhat less precise in that the differential produced could differ from $\phi^{-1}(v)$ by certain positive scaling factors on the various components.  Consequently our proof of the dimension bound will be more complicated.  

Let $\mathcal S$ be some decomposition of the set  of factor labels $\{1,\ldots,k\}$ into disjoint subsets $S_1 \sqcup \cdots \sqcup S_n$, where $S_1,\ldots,S_{n-1}$ are non-empty and $S_n$ consists of the factors where some (equivalently, all) differential in $\mathring{\partial} QN'$ is zero.  
We define maps 
$$B_\mathcal S: \mathring{\partial} QN' \to T\partial \M,$$
$$B_\mathcal S (q^1,\ldots,q^k) = \left(r_1 \phi(q^1), \ldots , r_k\phi(q^k)\right),$$
where $r_i$ is defined as follows.  If $i\in S_n$ then $r_i=1$, and otherwise $r_i:= \frac{1}{\|q^i\|} \sum_{j\in S_\ell} \|q^j\|$, with $\ell$ the index such that $i\in S_\ell$. 
Because we define the map only on $\mathring{\partial} QN'$, there is no division by zero.  The map is defined so that tangent vectors in the image have equal norms on the components lying in the same part of the decomposition.

\begin{lemma}
    \label{lem:Bsurj}
    The set 
    \[    
    \bigcup_\mathcal S B_\mathcal S (\mathring{\partial} QN') 
    \]
    contains a dense Zariski open subset of $T\mathring{\partial} N$.  Here the union ranges over the decompositions $\mathcal S$ where $B_\mathcal S$ is defined (see above).  
\end{lemma}

\begin{proof}
    We apply \Cref{prop:qd-seq}, with $v$ a generic tangent vector in $T\mathring{\partial} N$, yielding a $q\in \partial QN$ (on the same surface as $v$), and with $q^i=c_i \phi^{-1}(v_i)$ and $\sum_{j\in S_\ell} c_j = 1$ for all $\ell$.  
    We claim that $q\in \mathring{\partial} QN'$.  Note that since the $c_i$ are all positive, $q$ is zero exactly on those factors where $v$ is.  So, by \Cref{lem:tan-glob-inv}, \eqref{item:tang}, it suffices to show that $v$ belongs to the set $V'$ of vectors that are $0$ exactly on the factors where \textit{all} elements of $T\mathring{\partial} N$ are $0$.  Note that $V'$ is a Zariski open subset of $T\mathring{\partial} N$.  Since $T\mathring{\partial} N$ is irreducible, $V'$ is non-empty, and so a generic vector $v$ will lie in it, which gives the claim.  %

    Now we claim that $B_{\mathcal S(v)} (q)=v$.  First note that for $i\in S_n$, we have $B_{\mathcal S(v)}(q)^i =v^i$.  Now for $i$ in $S_\ell \ne S_n$, using that $\|v_i\|=\|v_j\|$ for any other $j\in S_\ell$, and that $\phi$ is norm-preserving, we get
    
    $$B_{\mathcal S(v)}(q)^i = r_i\phi(c_i \phi^{-1}(v^i)) = \frac{\sum_{j\in S_\ell} \|c_j\phi^{-1}(v_j)\| }{\|c_i \phi^{-1}(v^i)\|} c_i v^i = \frac{\|\phi^{-1}(v^i)\| \sum_{j\in S_\ell} c_j }{c_i \| \phi^{-1}(v^i)\|} c_i v^i= v^i,$$
    for each $i$, which establishes our claim.  
    
    The desired result then follows, since $\mathcal{S}(v)$ is one of the decompositions over which the union in the statement is taken.  
\end{proof}

\begin{lemma}
\label{lem:B-haus-dim}
For any $\mathcal S$ for which $B_\mathcal S$ is defined, and any $X\in \mathring{\partial} N$, we have
\[
\dim_H B_\mathcal S(\mathring{\partial} QN'_X) \le  \dim \partial QN_X',
\]
\end{lemma}
where $\dim_H$ denotes Hausdorff dimension.  

\begin{proof}

    We first show that $B_\mathcal S$ is a real-analytic map between varieties when restricted to each stratum $\mathcal{Q}(\mu)$. %
    On  $\mathcal{Q}(\mu)$, each norm function $\|q^j\|$ is real-analytic (for $j$ a component where the differentials are not identically zero), since it has a real-analytic expression in terms of period coordinates.  Next we claim that each component function $\phi(q^j)$ is real-analytic (as a function of $q^j$, and hence also as a function of $q$).  %
    On $\mathcal{Q}(\mu)$, the map $t\mapsto g_tq^j$ has a real-analytic expression in terms of period coordinates.  The tangent vector to this path is $\phi(q^j)/\|q^j\|$.  It follows that the component function $\phi(q^j)$ is real-analytic.  Since $B_\mathcal S$ is constructed from the $\phi(q^j)$ and norm functions, the restriction of $B_\mathcal S$  to $\mathcal{Q}(\mu)$ is real-analytic. 

    Real-analytic maps are smooth, and smooth maps are non-increasing for Hausdorff dimension.
    Hence 
    $$\dim_H B(\mathring{\partial} QN'_X\cap \mathcal{Q}(\mu)) \le \dim_H(\mathring{\partial} QN'_X\cap \mathcal{Q}(\mu)).$$  

    Applying this for each of the finitely many $\mu$ gives
    \begin{align*}        
    \dim_H B(\mathring{\partial} QN'_X) &= \max_\mu \dim_H B(\mathring{\partial} QN'_X\cap \mathcal{Q}(\mu)) \\
        &\le \max_\mu \dim_H(\mathring{\partial} QN'_X\cap \mathcal{Q}(\mu)) =\dim \mathring{\partial} QN'_X,
    \end{align*}
    as desired.  
\end{proof}

\begin{lemma}
    \label{lem:quad-dim}
   For generic $X\in \mathring{\partial} N$, we have 
   \[
   \dim \mathring{\partial} QN'_X \ge \dim T_X \mathring{\partial} N.
   \]
\end{lemma}

\begin{proof}

For generic $X$, by \Cref{lem:Bsurj}, we get that $\bigcup_\mathcal S B_\mathcal S (\mathring{\partial} QN'_X)$ contains a Zariski open subset of $T_X\mathring{\partial} N$.
Then taking dimensions and applying \Cref{lem:B-haus-dim} gives 
\begin{align*}
\dim T_X \mathring{\partial} N &\le \dim_H \left( \bigcup_\mathcal S B_\mathcal S (\mathring{\partial} QN'_X) \right) = \max_\mathcal S \dim_H B_\mathcal S(\mathring{\partial} QN'_X)  \\
    & \le \max_\mathcal S \dim \mathring{\partial} QN'_X = \dim \mathring{\partial} QN'_X. 
\end{align*}
\end{proof}

We have now assembled the tools that we need to prove the $GL^+(2,\R)$-geodesic property, the main result of this section.  

\begin{prop}
\label{prop:bdy-gl2r-geod}
Suppose $\M_{g,n}$ is a single moduli space, $N\subseteq \M_{g,n}$ a totally geodesic subvariety, and $\mathring{\partial} N$ an irreducible component of the intersection of $\partial N$ with a boundary stratum $\Delta_{\Gamma}$ of $\partial \M_{g,n}$.  Then $\mathring{\partial} N$ is $GL^+(2,\R)$-geodesic.  
\end{prop}

\begin{proof}
    Recall that, by \Cref{lem:tan-glob-inv} \eqref{item:alg}, $\mathring{\partial} QN'$ is algebraic, hence its intersection $\calQ(\mu)\cap \mathring{\partial} QN'$ with any stratum $\calQ(\mu)$ of the Hodge bundle boundary $\partial \calQ\overline{\calM}_{g,n}$ is also algebraic.  
    Let $Q_1,\ldots, Q_n$ be the (finitely many) irreducible components of $\mathcal Q(\mu)\cap \mathring{\partial} QN'$, ranging over the finitely many strata $\mathcal Q(\mu)$.  We will eventually choose $Q$ to be one of these $Q_i$.

    By \Cref{lem:quad-dim}, for generic $X\in \mathring{\partial} N$, we get $\dim \mathring{\partial} QN'_X \ge \dim T_X \mathring{\partial} N$
    and hence there exists some $i=i(X)$ such that $\dim (\mathring{\partial} QN'_X\cap Q_i) \ge \dim \mathring{\partial} N.$  
    
    Let 
    $$E_i := \{X\in \mathring{\partial} N: \dim (\mathring{\partial} QN'_X\cap Q_i)\ge \dim \mathring{\partial} N\}.$$ 
    By the previous paragraph $\mathring{\partial} N = \overline{E_1 \cup \cdots \cup E_n}$. Each $E_i$ is a constructible set, since the dimension of an algebraic set varies semi-continuously.  Now since $\mathring{\partial} N$ is an irreducible variety, $\mathring{\partial} N = \overline{E_i}$ for some $i$.  Pick some such $i$, and define $Q$ to be $Q_i$.  It has the property that $\dim (\mathring{\partial} QN'_X\cap Q) \ge \dim \mathring{\partial} N$ for generic $X\in \mathring{\partial} N$.  Since this dimension is non-zero, in particular we get that $\overline{p(Q)}=\mathring{\partial} N$, which is property \eqref{item:proj} of $GL^+(2,\R)$-geodesic.  We also get that $\dim Q \ge 2\dim \mathring{\partial} N$, property \eqref{item:dimQ}.  

    For property \eqref{item:non-vary}, first note that by the definition of $\mathring{\partial} QN'$, if the set $\mathcal Q(\mu) \cap \mathring{\partial} QN'$ is non-empty then for all factors where the differentials in $\calQ(\mu)$ are zero,
    it must be that $\mathring{\partial} N$ does not vary along that component.  Hence $Q$ does not vary along that component either, and we're done.

    For \eqref{item:gl2-inv}, we begin by applying \Cref{lem:gl2-inv}, which gives that $\mathring{\partial} QN'$ is $GL^+(2,\R)$-invariant.  Since each $\calQ(\mu)$ is $GL^+(2,\R)$-invariant, we get that $\calQ(\mu) \cap \mathring{\partial} QN'$ is as well.  Now $Q$ is some irreducible component of this variety, so for a generic point $q\in Q$, we have that $Q$ locally agrees with $\calQ(\mu) \cap \mathring{\partial} QN'$ within $\calQ(\mu)$.  Since $\calQ(\mu) \cap \mathring{\partial} QN'$ is $g_t$-invariant for all $t\in \R$, we get that for $t$ sufficiently small, $g_tq\in Q$.  Applying \Cref{lem:local-inv} gives that $Q$ is $g_t$ invariant for all $t\in \R$.  Invariance of $Q$ under $r_\theta$, for any $\theta\in S^1$, is immediate from $Q$ being a complex algebraic subvariety.  Since such $g_t, r_\theta$ generate $GL^+(2,\R)$, we get that $Q$ is $GL^+(2,\R)$-invariant.

\end{proof}

\section{$\GL^+(2,\RR)$-geodesic implies totally geodesic}
\label{sec:gl2r-geo-tot-geo}

The notion of $GL^+(2,\R)$-geodesic for a subset of multi-component moduli space $\M$ was introduced above in \Cref{sec:gl2r-geod}.  In this section, we show that this property implies infinitesimally totally geodesic, and then we combine with results from \Cref{sec:tot-geod} and  \Cref{sec:bdy-gl2r-geod} to prove the main result of the paper, \Cref{thm:bdy-tot-geod}.

\subsection{$\GL^+(2,\RR)$-geodesic implies infinitesimally totally geodesic}
The definition of $GL^+(2,\R)$-geodesic gives a large dimension set of quadratic differentials generating geodesics tangent to $N$.  However, what we really want in order to show the totally geodesic property is that the set of such tangent vectors has large dimension.  In multi-component moduli spaces, many different quadratic differentials can give rise to the same tangent vector, since rescaling the factors by different real numbers still results in the same \TM geodesic.  

To deal with this issue, we first consider the case of prime subvarieties (recall this notion was defined in \Cref{sec:prime}).  For these, we can apply the fact that ratios of areas are constant, \Cref{prop:ratio-areas-qd}, which allows us to rule out having too many quadratic differentials generating the same \TM geodesic.

\begin{defn}
    We say a geodesic in $\M$ (in the sense of \Cref{defn:geod}) is \emph{single-speed} if all the rescaling factors $c_i$ are the same. 
\end{defn}

Note that such a geodesic moves at the same speed in every factor of the product cover of $\M$.  

\begin{lemma}\label{lem:gl-single-speed}
Suppose $N\subseteq\calM$ is \emph{prime} and a $\GL^+(2,\RR)$-geodesic subvariety.  Then for a generic set of $X\in N$, for every $v\in T_XN$, there is a single-speed geodesic in $N$ and tangent to $v$.  In particular $N$ is infinitesimally totally geodesic.  
\end{lemma}

\begin{proof} 
For notational simplicity, we will assume that $\M$ is a product (otherwise, we can pass to its product cover).  Let $Q\subset \mathcal \calQ(\mu)$, $\mu=(\kappa_1,\ldots,\kappa_m)$, be the witness that $N$ is $GL^+(2,\R)$-geodesic.  
Since $Q$ is non-varying on any component $i$ for which $\kappa_i=\infty$, by removing these factors, we can regard $Q$ as a subvariety of a stratum $\mathcal Q (\tilde \mu)$ of multi-component \textit{translation} surfaces, i.e. none of the components of $\tilde \mu$ are $\infty$ (recall \Cref{sec:strata}).  From this, one gets a notion of prime, which $Q$ must satisfy, since a product decomposition of it would also yield a product decomposition of $N$. 

Let $X$ be a generic point of $N$, and let $Q_X$ denote the subset of $Q$ lying over $X$. Consider the map 
\begin{align}
    &\psi: Q_X\to T_XN, \\
    &\psi(q) = \frac{d}{dt} g_{t\|q\|_1}(q) |_{t=0}, \label{eq:psi}
\end{align}
i.e. $\psi(q)$ is the tangent vector to the path determined by $q$, moving at speed $\|q\|_1$ in every factor (recall $\|q\|_1$ is the area norm, given by the sum of the areas of the components). This tangent vector lies in $T_XN$ because $Q$ is $GL^+(2,\R)$-invariant, and $p(Q)\subset N$. (Recall that the $GL^+(2,\R)$-action is defined as the identity on any component where the the differential vanishes identically, and hence on these components $\psi(q)$ will be the zero vector). 

The relevance of the map $\psi$ is that for every $v\in \psi(Q),$ there is a single-speed geodesic lying in $N$ tangent to $v$. %
We would like to show that $\psi(Q)$ is large, and in particular, dense in $T_XN$.

Now consider the closure $\overline{Q_X}$ of $Q_X$ in the bundle $\calQ \M$.  For generic $X$, $\overline{Q_X}$ will be a smooth variety.   We can extend $\psi$ to a map on $\overline \psi$ on $\overline{Q_X}$, using the same formula as above.  A priori, $\overline \psi$ takes values in $T_X \M$.

We would like to show that $\overline{\psi}$ is continuous.  We have to be careful because $\psi$ extends naturally to a map on all of $\calQ\M$, but it is \textit{not} continuous there -- consider a sequence of unit area differentials that are non-vanishing on all components converging to a differential that vanishes on exactly one component.  In the claim below, we will use the prime property to rule this behavior out on $\overline{Q_X}$.

\begin{claim}
\label{claim:psi-cts}
    The map $\overline \psi$ is continuous, and takes values in $T_XN$.  
\end{claim}

\begin{proofclaim}
    To show continuity, we consider any sequence $q_n\to q$, with $q_n,q\in \overline{Q_X}$.  
    
    First consider the case $\|q\|_1=0$, i.e. $q$ is zero on every component of $X$.  Since area varies continuously on $\calQ \M$, we see that $\|q_n\|_1\to 0$, and thus $\overline\psi(q_n) \to 0=\overline \psi (q).$  So this case is complete.  

    Now consider the case $\|q\|_1 \ne 0$.  By the homogeneity properties of $\overline{Q_X}$ and $\overline \psi$, it suffices to assume that $\|q\|_1=\|q_n\|_1=1$, for all $n$. %
    Now all these $q_n,q$ are limits of unit norm elements of $Q_X$. %
    By \Cref{prop:ratio-areas-qd} (applied in $\calQ (\tilde \mu)$), the ratios of areas of different components of differentials in the unit area locus of $Q$ are constant.  Since area varies continuously, all the $q_n,q$ also have these same ratios of areas.  In particular, each $q_n,q$ vanishes on some component of $X$ iff the differentials in the stratum $\calQ(\mu)$ do.  For such sequences, convergence of the values $\overline \psi(q_n)$ can be seen as in the proof of \Cref{lem:homeo} giving continuity of $\phi$ (which is related to $\overline\psi$ by rescaling by area norm).  We have thus completed this case, and hence we conclude that $\overline \psi $ is continuous.  

    Finally, since $\overline \psi$ is continuous and the values of $\psi$ lie in the closed subset $T_XN$ of $T_X\M$, it follows that values of $\overline \psi$ also lie in $T_XN$.  
\end{proofclaim}

    The claim above allows us to think of $\overline{\psi}$ as a map $\overline{Q_X} \to T_XN$, which we will do for the remainder of the proof.

\begin{claim}
\label{claim:psi-inj}
    The map $\overline\psi$ is injective.  
\end{claim}

\begin{proofclaim}
    Suppose that $q_1,q_2 \in \overline{Q_X}$ with $\psi(q_1)=\psi(q_2)$.  Let $q_j=(q_j^1,\ldots,q_j^k)$ for $j=1,2$, and $\psi(q_1)=\psi(q_2)=(v^1,\ldots,v^k)$.   Note that for each $i$, $q_1^i$ and $q_2^i$ are equal up to some constant real multiple (cf. \Cref{subsec:phi}).  Also, for $j=1,2$ and any $i$ such that $q_j^i\ne 0$, we have $\|v^i\|=\|q_j\|_1$  By considering such an $i$, we see that $\|q_1\|_1 = \|q_2\|_1$, i.e.  they have the same total area.

    Each $q_j$ is a limit of differentials in $Q$, and, by rescaling, we can take these differentials to also have norm $\|q_1\|_1 = \|q_2\|_1$.  Now by \Cref{prop:ratio-areas-qd} (applied in $\calQ (\tilde \mu)$), the ratios of areas of different components of differentials in each constant area locus of $Q$ are constant. %
    So by continuity of area, it then follows that $\|q_1^i\| = \|q_2^i\|$ for all $i$.  Since we also know that $q_1^i$ and $q_2^i$ are equal up to some constant real multiple, we now get that this multiple must be $1$.  So $q_1^i=q_2^i$ for all $i$, i.e. $q_1=q_2$, as desired.  

\end{proofclaim}

Using the irreducibility of $Q$, $\dim Q \ge 2\dim N$, and $\overline{p(Q)}=N$, we get $\dim Q_X\ge \dim N$ for generic $X$.  Thus $\overline \psi$ is a continuous (\Cref{claim:psi-cts}) and injective (\Cref{claim:psi-inj}) map from $\overline{Q_X}$, which is a smooth variety of dimension at least $\dim N$, to $T_XN$.  By invariance of domain, its image must be open.  Now the quadratic differential on $X$ that is zero on all factors is contained in $\overline{Q_X}$, and hence the zero vector is in the image of $\overline{\psi}$.  Since $\overline \psi$ is homogeneous with respect to $\R$ scaling, its image must be similarly homogeneous.  Since the image is also open and contains $0$, it must in fact be all of $T_XN$.  

Since $\overline{Q_X}-Q_X$ has smaller dimension than $Q_X$, we deduce from the above that $\psi(Q_X)$ is topologically dense in $T_XN$.  As discussed above, all the vectors in this image are tangent to single-speed geodesics lying in $N$.  Since the set of vectors with this property is topologically closed (since $N$ is closed), we've proved the desired result.

\end{proof}

We now want to relax the condition that $N$ be prime.  We will decompose into prime factors.  We need to understand how the witness to the $GL^+(2,\R)$ property decomposes.  For this we will use the following.  

\begin{lemma}
\label{lem:dim-upper-bound} 
Suppose $Q\subseteq \calQ(\mu)$, $\mu=(\kappa_1,\ldots, \kappa_m)$, is a $\GL^+(2,\RR)$-invariant, irreducible, algebraic set, and that $Q$ does not vary on any component $i$ for which $\kappa_i=\infty$.   Let  $N = \overline{p(Q)}$. Then
\[
\dim Q \leq 2\dim N.
\]
\end{lemma}

\begin{proof}

From the non-varying assumption, we can regard $Q$ as a subvariety of a stratum $\calQ(\tilde \mu)$ of multi-component translation surfaces, i.e. none of the components of $\tilde \mu$ are $\infty$.  And $Q$ is prime in $\calQ (\tilde \mu)$ iff it is prime in $\calQ (\mu)$.  

We first consider the case when $Q$ is prime.  Let $X$ be a smooth point of $N$.  
Consider the map $\psi: Q_X\to T_XN$ defined in \eqref{eq:psi} in proof of \Cref{lem:gl-single-speed}.  As discussed there, $\psi$ is continuous and injective (using that $Q$ is prime in $\calQ (\tilde \mu)$).

It follows that $\dim Q_X \le \dim T_X N = \dim N$, for $X\in N^{sm}$ a smooth point.  Now
\[
\dim p^{-1}(N^{sm}) \le \dim N^{sm} + \max_{X\in N^{sm}} \dim Q_X \le \dim N + \dim N.
\]
Now since $Q=p^{-1}(N)$ is irreducible, it equals the closure of $p^{-1}(N^{sm})$.  Hence we get $\dim Q \le 2 \dim N$, so we are done in the case $Q$ prime.

Now if $Q$ is not prime, then $Q$ is a product of primes $Q=Q_1 \times \cdots \times Q_k$. We see that each $Q_i$ is $GL^+(2,\R)$-invariant, irreducible, algebraic, and does not vary on any component for which $\kappa_i=\infty$. Set $N_i = \overline{p(Q_i)}$. Then $N=N_1\times \cdots \times N_k$ and, by the result for primes,
\[
\dim Q = \dim Q_1 + \cdots + \dim Q_k \leq 2\dim N_1 + \cdots + 2\dim N_k =2\dim N,
\]
and we are done.

\end{proof}

Using the above, we next show that the prime factors are also $GL^+(2,\R)$-geodesic.  

\begin{prop}
\label{prop:prod-gl2r}
Let $N \subset \M_{g_1,n_1} \times \cdots \times \M_{g_k,n_k}$ be a $GL^+(2,\R)$-geodesic subvariety.  Suppose $N = N_1 \times N_2$, where $N_1\subset \M_{g_1,n_1} \times \cdots \times \M_{g_s,n_s}$ and $N_2\subset \M_{g_{s+1},n_{s+1}} \times \cdots \times \M_{g_k,n_k}$, for some $k$.  Then each $N_i$ is also $GL^+(2,\R)$-geodesic.  
\end{prop}

\begin{proof}
    Let $Q$ be a witness that $N$ is $GL^+(2,\R)$-geodesic.  Let $\rho_i: N\to N_i$ be the projection.  We claim that $\overline{\rho_i(Q)}$ is a witness that $N_i$ is $GL^+(2,\R)$-geodesic.  First note that it is irreducible, as it's the (closure of the) image under an algebraic map of an irreducible algebraic set.  

    Properties \eqref{item:gl2-inv}, \eqref{item:proj}, \eqref{item:non-vary} for $\rho_i(Q)$ follow immediately from the corresponding properties of $Q$.  

    For property \eqref{item:dimQ}, since $Q\subset \rho_1(Q) \times \rho_2(Q)$ and using property \eqref{item:dimQ} for $Q$, we get
    \begin{align}
        \dim \rho_1(Q) + \dim \rho_2(Q) \ge \dim Q \ge 2\dim N = 2\dim N_1 + 2 \dim N_2.
    \end{align}
    Now by \Cref{lem:dim-upper-bound}, $\dim \rho_i(Q) \le 2 \dim N_i$ for each $i$.  Combined with the above, this implies that in fact $\dim \rho_i(Q) = 2 \dim N_i$ for each $i$, which is property \eqref{item:dimQ}.  
\end{proof}

Finally, we can use the above to relax the condition that $N$ be prime.  

\begin{lemma}\label{lem:gl-to-inf}
Suppose $N\subseteq\calM$ is any $\GL^+(2,\RR)$-geodesic subvariety. Then $N$ is infinitesimally totally geodesic.
\end{lemma}

\begin{proof}

Note that the product cover of $N$ is also $GL^+(2,\R)$-geodesic (a witness can be lifted to quadratic differentials over the product cover).  We begin by decomposing this product cover as 
\[
\prod_i N_i
\]
where $N_i$ are prime.  By \Cref{prop:prod-gl2r} (and analogous versions of it in which the components are reordered), each $N_i$ is $GL^+(2,\R)$-geodesic.  So we can apply \Cref{lem:gl-single-speed}, which yields that each $N_i$ is infinitesimally totally geodesic.  It is easy to see that a product of infinitesimally geodesic subvarieties is again infinitesimally totally geodesic.  So the product cover of $N$ is infinitesimally totally geodesic.  This implies that $N$ itself is infinitesimally totally geodesic.  

\end{proof}

\subsection{Proof of \Cref{thm:bdy-tot-geod}}

We can now prove the main result of the paper.  

\begin{proof}[Proof of \Cref{thm:bdy-tot-geod}]
    By \Cref{prop:bdy-gl2r-geod}, each $\mathring{\partial} N$ is $GL^+(2,\R)$-geodesic. Then \Cref{lem:gl-to-inf} gives that $\mathring{\partial} N$ is infinitesimally totally geodesic, and finally by \Cref{prop:inf-geod-tot-geod}, it's totally geodesic.

\end{proof}

\section{Structure of prime totally geodesic subvarieties}
\label{sec:bdy-struc}

In this section we prove \Cref{thm:bdy-struc}, using results from the previous sections.  We will first prove two preliminary lemmas, starting with the following which gives injectivity of the projection maps on the lift to \TM space.   

\begin{lemma}
    \label{lem:teich-single-speed}
    Let $N\subset \M_{g_1,n_1} \times \cdots \times \M_{g_k,n_k}$ be prime and $GL^+(2,\R)$-geodesic.  Let $\tilde N$ be a lift of $N$ to $\T:=\T_{g_1,n_1} \times \cdots \times \T_{g_k,n_k}$.  Then for any distinct points $X,Y\in \tilde N$ there is a unique single-speed geodesic lying in $\tilde N$ that contains both $X,Y$.  Furthermore, for each $i$, the projection map $\tilde N \to \T_{g_i,n_i}$ is injective.  
\end{lemma}

\begin{proof}
    By \Cref{lem:gl-single-speed}, for a generic point in $N$ and any tangent vector at that point, there is single-speed geodesic in $N$ tangent to that vector.  By lifting to $\tilde N$, we get that, for a generic $X$ in $\tilde N$, for all $v\in T_X\tilde N$ there is a single-speed geodesic tangent to $v$ and lying in $\tilde N$.  As in the proof of \Cref{prop:inf-geod-tot-geod}, we can deduce from this that any $X,Y\in \tilde N$ can be connected by a single-speed geodesic lying in $\tilde N$.  

    The uniqueness of the geodesic follows from uniqueness of geodesics in $\T$.  

    The injectivity statement follows from the existence of the single-speed geodesic: if $X,Y$ are distinct, then they differ in some factor, but since there is a single-speed geodesic connecting them, they must then differ in all factors. 
\end{proof}

\begin{proof}[Proof of \Cref{thm:bdy-struc}]
    By \Cref{prop:bdy-gl2r-geod}, $\mathring{\partial} N$ is $GL^+(2,\R)$-geodesic. It follows easily that its product cover is also $GL^+(2,\R)$-geodesic.  By \Cref{prop:prod-gl2r}, each prime factor $N_j$ of the product cover of $\mathring{\partial} N$ in some $\M_{g_1,n_1} \times \cdots \times \M_{g_k,n_k}$ is also $GL^+(2,\R)$-geodesic.  
    Now \Cref{lem:teich-single-speed} implies that for a lift $\tilde N_j$ of $N_j$ to $\T:=\T_{g_1,n_1} \times \cdots \times \T_{g_k,n_k}$, each projection map $\tilde \Phi: \tilde N_j \to \T_{g_i,n_i}$ is injective.  
    This means that the map $\Phi:N_j \to \M_{g_i,n_i}$ is locally injective in the orbifold sense, as desired.

\end{proof}

Next we show the stronger statement that the projections in the above theorem are in fact local isometries, in an appropriate sense.  We consider the metric $d_p$ from \Cref{defn:Lp-tot-geod}, for $1\le p \le \infty$.  Suppose $N\subset \M$ a totally geodesic subvariety, and $\Phi:N \to \M'$, where $\M,\M'$ are multi-component moduli spaces, with \TM space covers $\T,\T'$, respectively.  We say that $\Phi$ is a $d_p$-\emph{local isometry} if for any lift $\tilde N\subset \T$ of $N$, any lift $\tilde \Phi: \tilde N \to \T'$ is a local isometry for $d_p$ (this metric is used on domain and target).  

(The reason we work with lifts to \TM spaces rather than directly on multi-component moduli spaces is due to potential issues at orbifold points.)  

\begin{prop}
\label{prop:local-isom}
    With the same setup as \Cref{thm:bdy-struc}, for any for $1<p<\infty$, the map $\Phi$ is a $d_p$-local isometry.  
\end{prop}

\begin{proof}
    As in proof of \Cref{thm:bdy-struc}, we see that each $N_j$ is $GL^+(2,\R)$-geodesic.  Consider any lift $\tilde N_j$ of $N_j$ to $\T:=\T_{g_1,n_1} \times \cdots \times \T_{g_k,n_k}$.  We claim that $\tilde \Phi: \tilde N_j \to \T_{g_i,n_i}$ is a $d_p$-isometry, which implies $\Phi$ is a $d_p$-local isometry.  Let $X,Y\in \tilde N_j$ be distinct. By \Cref{lem:teich-single-speed}, there exists a single-speed geodesic arc $\gamma$ connecting $X,Y$.  This means the projections of $\gamma$ to each of the factors of $\T$ are all \TM geodesic arcs of the same length $L$. Since the $d_p$ metric is uniquely geodesic and $\gamma$ is a geodesic arc for this metric (see proof of \Cref{prop:tot-geod-equiv}), it follows that $d_p(X,Y)=L$.  But $\tilde\Phi(X), \tilde\Phi(Y)\in \T_{g_i,n_i}$ are connected by $\Phi(\gamma)$, which we've seen also has length $L$, and hence $d_p(\tilde\Phi(X), \tilde\Phi(Y))=L$, and we are done.  
    
    \end{proof}

\begin{example}    
The following example shows that local injectivity, in the orbifold sense, of the projection map in \Cref{thm:bdy-struc} cannot be upgraded to injectivity, for general $GL^+(2,\mathbb{R})$-geodesic subvarieties (we do not claim that our example appears as the boundary of a totally geodesic subvariety of some $\M_{g,n}$).  

It is known \cite{schmithusen06} that for any genus $g \geq 2$, there exist surfaces of genus $g$ with Veech group $\Gamma(2)$, the principal congruence subgroup of $SL(2,\mathbb{Z})$ of level $2$. Let $N_1 \subset \calM_{g}$ be the associated Teichm{\"u}ller curve for such a surface, and let $N_2$ be the modular curve (stabilizer $SL(2,\mathbb{Z})$). The index is $[\Gamma:\Gamma(2)] = 6$. Thus, there is a degree six covering map $f:N_1 \to N_2$, which is a local isometry. Now, define the totally geodesic prime subvariety $N \subset \mathcal{M}_{g} \times \calM_{1,1}$ to be the graph of $f$ as a map from $N_1$ to $N_2$. Since $f$ is not injective, the projection $\rho_2: N \to \calM_{1,1}$ is not injective.

We note that $N$ is $GL^+(2,\mathbb{R})$-geodesic in the sense of Definition \ref{defn:gl2r-geod}. Indeed, consider the following set, defining a witness

\begin{displaymath}
Q = \left\{(q,s_{*}q) : q \in QN_1 \right\} \subset \mathcal{Q} \mathcal{M}_g \times \mathcal{Q} \mathcal{M}_{1,1}.
\end{displaymath}

The pushforward $s_{*}q$ is the pushforward of a torus cover as a half-translation surface, to the torus. Observe that $N$ is prime, as $Q$ is complex 2-dimensional, so cannot be written as a product of a subset in $\mathcal{Q}\mathcal{M}_g$ and a subset in $\mathcal{Q}\mathcal{M}_{1,1}$.
\end{example}

{\footnotesize
\bibliographystyle{amsalpha}
  \bibliography{sources}%
  }

\end{document}